\documentclass[a4paper, 10pt]{amsart}
\usepackage[latin1]{inputenc}
\usepackage[francais]{babel}
\usepackage{amsfonts}
\usepackage{amsmath}
\usepackage{amssymb}
\usepackage{amsthm}
\usepackage{amsxtra}
\usepackage{amstext}
\usepackage{eufrak}
\usepackage{latexsym}
\usepackage[T1]{fontenc}
\usepackage[all]{xy}

\def\FF{\mathbb{F}}
\def\FFF{\overline{\mathbb{F}}}
\def\NN{\mathbb{N}}
\def\PP{\mathbb{P}}
\def\QQ{\mathbb{Q}}
\def\QQQ{\overline{\mathbb{Q}}}

\newtheorem{thm}{Th\'eor\`eme}
\newtheorem{cor}{Corollaire}
\newtheorem{prop}{Proposition}
\newtheorem{lm}{Lemme}
\newtheorem{rem}{Remarque}
\newenvironment{NB}{\begin{rem} \rm }{\end{rem}}

\title[Repr\'esentations modulo $p$ de $SL_{2}(F)$]{Une \'etude des repr\'esentations modulo $p$ de $SL_{2}(F)$}
\author{Ramla ABDELLATIF}
\address{Universit\'e Paris-Sud 11, 91 405 Orsay Cedex, ramla.abdellatif@math.u-psud.fr}
\date{}
\begin{document}
\begin{abstract}
A la suite des travaux de Barthel-Livn\'e \cite{BL1, BL2} et Breuil \cite{Br1}, nous \'etudions les repr\'esentations modulo $p$ de $SL_{2}(F)$ lorsque $F$ est un corps local complet de caract\'eristique r\'esiduelle $p$ et de corps r\'esiduel fini. En particulier, nous les relions aux repr\'esentations modulo $p$ de $GL_{2}(F)$ et, lorsque $F = \QQ_{p}$, nous donnons une description explicite des repr\'esentations supersinguli\`eres de $SL_{2}(\QQ_{p})$, qui apparaissent par paquets de taille 1 ou 2.
\end{abstract}

\maketitle
\tableofcontents

\section*{Introduction}
Soient $p$ un nombre premier et $F$ un corps local non archim\'edien de caract\'eristique r\'esiduelle $p$. Dans les ann\'ees 90, Barthel et Livn\'e \cite{BL1} ont donn\'e une classification des repr\'esentations modulo $p$ de $GL_{2}(F)$ faisant ainsi appara\^itre une famille de repr\'esentations dites supersinguli\`eres dont la description explicite est inconnue dans le cas g\'en\'eral. Le seul cas bien connu est celui o\`u $F = \QQ_{p}$, dans lequel les travaux de Breuil \cite{Br1} permettent de d\'ecrire compl\`etement les repr\'esentations supersinguli\`eres de $GL_{2}(\QQ_{p})$ et d'obtenir ce qu'il appelle une << correspondance de Langlands locale semi-simple modulo $p$ >> \cite[D\'efinition 4.2.4]{Br1}.\\
\indent L'objectif de ce papier est d'\'etudier les repr\'esentations modulo $p$ de $SL_{2}(F)$. En effet, le groupe $SL_{2}(F)$ est assez proche de $GL_{2}(F)$ pour que l'on puisse esp\'erer obtenir des r\'esultats analogues, mais suffisamment diff\'erent pour que l'on voie d\'ej\`a appara\^itre quelques divergences pouvant peut-\^etre aider \`a comprendre ce qui bloque dans l'\'etude des repr\'esentations modulo $p$ de $GL_{2}(F)$ lorsque $F$ n'est pas $\QQ_{p}$, voire de $GL_{n}(F)$.

\subsection*{Pr\'esentation des principaux r\'esultats}
Nous commen\c{c}ons par d\'emontrer un r\'esultat de classification analogue \`a \cite[page 290]{BL1} lorsque $F$ est de caract\'eristique diff\'erente de $2$.
\begin{thm}
Supposons que $F$ soit de caract\'eristique diff\'erente de $2$.
\begin{enumerate}
\item Les classes d'isomorphisme des repr\'esentations lisses irr\'eductibles \`a caract\`ere central de $SL_{2}(F)$ sur $\FFF_{p}$ se partitionnent en quatre familles :
\begin{enumerate}
\item le caract\`ere trivial ;
\item les s\'eries principales ;
\item les s\'eries sp\'eciales ;
\item les supersinguli\`eres.
\end{enumerate}
\item Il n'existe pas d'isomorphisme entre repr\'esentations provenant de familles diff\'erentes.
\end{enumerate}
\label{nonssg}
\end{thm}
\noindent L'hypoth\`ese sur la caract\'eristique de $F$ est une hypoth\`ese technique (cf. Section \ref{restriction}). Nous pensons cependant que le Th\'eor\`eme \ref{nonssg} est valable en caract\'eristique quelconque.\\

\noindent Nous prouvons au passage le r\'esultat suivant, valable en toute caract\'eristique, qui permet  de comparer les repr\'esentations non supersinguli\`eres de $SL_{2}(F)$ \`a celles de $GL_{2}(F)$.\begin{thm}
\begin{enumerate}
\item Si $\eta : F^{\times} \to \FFF_{p}^{\times}$ d\'efinit une repr\'esentation en s\'eries principales de $G$, alors on a un isomorphisme naturel de $G_{S}$-repr\'esentations :
$$ Ind_{B}^{G}(\eta \otimes \mathbf{1}) \simeq Ind_{B_{S}}^{G_{S}}(\eta) \ .$$
\item On a un isomorphisme naturel de $G_{S}$-repr\'esentations : 
$$ Sp \simeq Sp_{S} \ . $$
\item La restriction \`a $G_{S}$ de la repr\'esentation $Ind_{B}^{G}(\mathbf{1})$ est naturellement isomorphe \`a $Ind_{B_{S}}^{G_{S}}(\mathbf{1})$.
\end{enumerate}
\label{dejairred}
\end{thm}

Le Th\'eor\`eme \ref{nonssg} permet aussi d'obtenir facilement l'\'equivalence entre les notions de supercuspidalit\'e et de supersingularit\'e pour les repr\'esentations lisses irr\'eductibles de $SL_{2}(F)$.
\begin{cor}
Une repr\'esentation lisse irr\'eductible de $SL_{2}(F)$ sur $\FFF_{p}$ est supersinguli\`ere si et seulement si elle est supercuspidale.
\label{ssgscp}
\end{cor}

Lorsque $F = \QQ_{p}$, on utilise les r\'esultats de Breuil \cite{Br1} pour construire explicitement des repr\'esentations supersinguli\`eres $\pi_{r, \infty}$ et $\pi_{r, 0}$ de $SL_{2}(\QQ_{p})$ (avec $0 \leq r \leq p-1$) qui permettent de prouver le th\'eor\`eme suivant.
\begin{thm}
On suppose que $F = \QQ_{p}$.
\begin{enumerate} 
\item Les repr\'esentations supersinguli\`eres de $SL_{2}(\QQ_{p})$ sont exactement (\`a isomorphisme pr\`es) les repr\'esentations $\pi_{r, \infty}$ et $\pi_{r, 0}$ avec $r \in \{0, \ldots, p-1\}$. 
\item Les repr\'esentations $\pi_{r, \infty}$ sont deux \`a deux non isomorphes pour $r$ parcourant $\{0, \ldots, p-1\}$. De m\^eme, les repr\'esentations $\pi_{r,0}$ sont deux \`a deux non isomorphes pour $r$ parcourant $\{0, \ldots, p-1\}$.
\item Les repr\'esentations $\pi_{r, 0}$ et $\pi_{s, \infty}$ sont isomorphes si et seulement si $s + r = p-1$.
\item Toute repr\'esentation supersinguli\`ere de $GL_{2}(\QQ_{p})$ restreinte \`a $SL_{2}(\QQ_{p})$ se d\'ecompose en somme directe de deux repr\'esentations supersinguli\`eres de $SL_{2}(\QQ_{p})$.
\end{enumerate}
\label{ssgqp}
\end{thm}

\subsection*{Plan de l'article}
Dans la Section \ref{rcl}, nous rappelons tout d'abord quelques r\'esultats et notations apparaissant dans l'\'etude men\'ee dans \cite{BL1, BL2, Br1} pour $GL_{2}(F)$, puis nous donnons les notations et les r\'esultats correspondants pour $SL_{2}(F)$.\\
Dans la Section \ref{restriction}, nous d\'emontrons un r\'esultat analogue \`a  \cite[Proposition 2.2]{Hen} lorsque $F$ est de caract\'eristique diff\'erente de $2$. Il permet notamment de justifier pourquoi l'\'etude des repr\'esentations supersinguli\`eres effectu\'ee dans la Section \ref{cassg} se ram\`ene \`a l'\'etude des restrictions \`a $SL_{2}(\QQ_{p})$ des repr\'esentations supersinguli\`eres de $GL_{2}(\QQ_{p})$.\\
La Section \ref{classifnssg} sera consacr\'ee \`a la preuve du Th\'eor\`eme \ref{dejairred} puis du Th\'eor\`eme \ref{nonssg}, tandis que le Th\'eor\`eme \ref{ssgqp} sera d\'emontr\'e dans la Section \ref{cassg}, o\`u nous dirons quelques mots sur ce qui advient de la correspondance de Langlands modulo $p$ obtenue par Breuil dans le cas de $GL_{2}(\QQ_{p})$.

\subsection*{Quelques remarques}

\begin{NB}
Signalons une fois pour toutes que la torsion d'une repr\'esentation lisse de $GL_{2}(F)$ par un $\FFF_{p}$-caract\`ere lisse de $GL_{2}(F)$ (qui est n\'ecessairement de la forme $\chi \circ \det$ avec $\chi : \QQ_{p}^{\times} \to \FFF_{p}^{\times}$ caract\`ere lisse) ne modifie clairement pas la repr\'esentation de $SL_{2}(F)$ obtenue par restriction.
\end{NB}

\begin{NB} Mentionnons ici que, durant la r\'edaction de ce papier, C. Cheng a d\'emontr\'e de mani\`ere ind\'ependante une partie des r\'esultats que nous prouvons ici \cite{CC}. Nous le remercions de nous avoir communiqu\'e ses notes.\\
\end{NB}

\noindent Nous tenons aussi \`a remercier G. Henniart, sous la direction de qui ce travail a \'et\'e effectu\'e, pour son int\'er\^et et ses remarques pertinentes sur une version pr\'eliminaire de cet article.

\section{Rappels et compl\'ements}
\label{rcl}
\subsection{Notations g\'en\'erales}
On fixe un nombre premier $p$ ainsi qu'un corps local complet $F$ de caract\'eristique r\'esiduelle $p$ et de corps r\'esiduel fini. On d\'esigne par $\mathcal{O}_{F}$ l'anneau des entiers de $F$, on en fixe une fois pour toutes une uniformisante $\varpi_{F} \in \mathcal{O}_{F}$ et on note $k_{F} = \mathcal{O}_{F}/(\varpi_{F})$ le corps r\'esiduel de $F$ : c'est une extension finie (de degr\'e not\'e $f$) du corps  \`a $p$ \'el\'ements $\FF_{p}$, dont on fixe une cl\^oture alg\'ebrique $\FFF_{p}$ contenant $k_{F}$. On fixe aussi un plongement $\iota : k_{F} \hookrightarrow \FFF_{p}$ et on note $v_{F}$ la valuation $\varpi_{F}$-adique de $F$ que l'on a normalis\'ee par $v_{F}(\varpi_{F}) = 1$.\\

Sauf mention contraire explicite, toutes les repr\'esentations consid\'er\'ees dans la suite seront \`a coefficients dans $\FFF_{p}$.\\ 

\subsection{Rappels de th\'eorie des repr\'esentations} 
Soient $H$ un groupe topologique et $\pi$ une repr\'esentation de $H$ sur un $\FFF_{p}$-espace vectoriel $V$.\\
On dit que $\pi$ est \emph{irr\'eductible} si $V \not= \{0\}$ et si les seuls sous-espaces vectoriels de $V$ stables sous l'action de $H$ sont $\{0\}$ et $V$.\\
On dit que $\pi$ est une repr\'esentation \emph{lisse} de $H$ si pour tout \'el\'ement $v$ de $V$, le sous-groupe Stab$_{H}$($v$) des \'el\'ements de $H$ fixant $v$ est un sous-groupe ouvert de $H$.\\
On dit que $\pi$ est une repr\'esentation \emph{admissible} lorsqu'elle est lisse et que, pour tout sous-groupe ouvert $K$ de $H$, l'espace $V^{K}$ des vecteurs de $V$ fixes sous l'action de $K$ est un $\FFF_{p}$-espace vectoriel de dimension finie.\\
On dit que $\pi$ est un \emph{caract\`ere modulo $p$ de $H$} lorsque $V$ est de dimension 1 sur $\FFF_{p}$.\\
Enfin, si $Z_{H}$ d\'esigne le centre de $H$, et si $\eta$ est un caract\`ere modulo $p$ de $Z_{H}$, on dit que $\pi$ \emph{admet pour caract\`ere central $\eta$} si l'action de $Z_{H}$ sur $V$ est donn\'ee par le caract\`ere $\eta$.\\

Le r\'esultat que nous rappelons maintenant  \cite[Lemma 3]{BL1} est un fait crucial de la th\'eorie des repr\'esentations modulo $p$ dont nous nous servirons constamment.
\begin{prop}
Soient $P$ un pro-$p$-groupe et $V$ une repr\'esentation lisse non nulle de $P$ sur $\overline{\mathbb{F}_{p}}$. Alors $V$ contient un vecteur fixe non trivial sous l'action de $P$.
\label{faitcrucial}
\end{prop}

Enfin, pour tout scalaire $\lambda \in \FFF_{p}^{\times}$, on note $\mu_{\lambda} : F^{\times} \to \FFF_{p}^{\times}$ le caract\`ere non ramifi\'e (i.e. trivial en restriction au groupe des unit\'es $\mathcal{O}_{F}^{\times}$ de $F$) envoyant $\varpi_{F}$ sur $\lambda$.

\subsection{El\'ements et sous-groupes remarquables de $GL_{2}(F)$}
Les notations introduites ici sont essentiellement celles utilis\'ees dans \cite{Br1} et \cite{Oll1}.\\ 
On consid\`ere le groupe $G:=GL_{2}(F)$, dont $K:=GL_{2}(\mathcal{O}_{F})$ est \`a conjugaison pr\`es l'unique sous-groupe compact maximal, et dont le centre est $Z := F^{\times}\*I_{2}$ avec $I_{2} := \left( \begin{array}{cc} 1 & 0 \\ 0 & 1\end{array}\right)$. On note $I \subset K$ le sous-groupe d'Iwahori standard et $I(1)$ son unique pro-$p$-sous-groupe de Sylow (que l'on appelle le pro-$p$-Iwahori de $K$). Ils correspondent respectivement \`a l'image r\'eciproque par l'application de r\'eduction modulo $\varpi_{F}$ du sous-groupe des matrices triangulaires sup\'erieures et du sous-groupe des matrices triangulaires sup\'erieures unipotentes (i.e. dont les coefficients diagonaux sont tous \'egaux \`a 1) de $GL_{2}(k_{F})$.\\
On note $B$ le sous-groupe de Borel de $G$ constitu\'e des matrices triangulaires sup\'erieures, $T$ le tore maximal d\'eploy\'e des matrices diagonales et $U$ le radical unipotent de $B$. Rappelons que $B$ admet alors la d\'ecomposition $B = TU$.\\

\noindent On d\'efinit les matrices suivantes :
$$ \alpha := \left( \begin{array}{cc} 1 & 0 \\ 0 & \varpi_{F} \end{array}\right) \ ; \  \beta := \left( \begin{array}{cc} 0 & 1 \\ \varpi_{F} & 0 \end{array}\right) \ ; \  \omega := \left( \begin{array}{cc} 0 & 1 \\ 1 & 0 \end{array}\right) \ .$$
Elles satisfont en particulier l'\'egalit\'e $\beta\omega = \alpha$.\\

\subsection{Poids de Serre de $GL_{2}(F)$}
\label{SerreGL}
On rappelle que l'on a fix\'e une fois pour toute une uniformisante $\varpi_{F}$ de $F$ car ce qui suit d\'epend du choix de cette uniformisante.\\
Pour tout entier $r \in \{0, \ldots , p-1\}$, on note $Sym^{r}(\FFF_{p}^{2})$ (ou $\sigma_{r}$) l'unique repr\'esentation de $KZ$ ayant pour espace sous-jacent $\displaystyle \bigoplus_{i = 0}^{r} \FFF_{p}x^{r-i}y^{i}$ sur lequel $\varpi_{F}I_{2}$ agit trivialement et sur lequel  $\left(\begin{array}{cc} a & b \\ c & d \end{array}\right) \in K$ agit par\footnote{Dans le membre de droite, les coefficients sont les r\'eductions modulo $\varpi_{F}$ des coefficients de la matrice de $K$.} : 
$$ \sigma_{r}(\left( \begin{array}{cc} a & b \\ c & d \end{array}\right))\*(x^{r-i}\*y^{i}) := (a\*x + c\*y)^{r-i}\*(b\*x + d\*y)^{i} \ .$$

\noindent Notons $q = p^{f}$ le nombre d'\'el\'ements de $k_{F}$. Pour tout $f$-uplet d'entiers $\vec{r}:=$($r_{0}, \ldots, r_{f-1})$ de $\{0,\ldots, p-1\}^{f}$, on d\'efinit $\sigma_{\vec{r}}$ comme \'etant l'unique repr\'esentation de $KZ$ ayant pour espace sous-jacent
$$ V_{\vec{r}} := Sym^{r_{0}}(\FFF_{p}^{2}) \displaystyle \otimes_{\FFF_{p}} Sym^{r_{1}}(\FFF_{p}^{2}) \otimes_{\FFF_{p}} \ldots \otimes_{\FFF_{p}} Sym^{r_{f-1}}(\FFF_{p}^{2})$$
sur lequel $\varpi_{F}I_{2}$ agit trivialement et sur lequel  $\left(\begin{array}{cc} a & b \\ c & d \end{array}\right) \in K$ agit par
$$\sigma_{\vec{r}}(\left(\begin{array}{cc} a & b \\ c & d \end{array}\right))(v_{0}\otimes \ldots v_{f-1}) := \displaystyle \bigotimes_{j = 0}^{f-1} \sigma_{r_{j}}(\left(\begin{array}{cc} a^{p^{j}} & b^{p^{j}} \\ c^{p^{j}} & d^{p^{j}} \end{array}\right))(v_{j})$$

A isomorphisme et torsion par un caract\`ere lisse pr\`es, les $\sigma_{\vec{r}}$ donnent toutes les repr\'esentations lisses irr\'eductibles de $KZ$ sur $\FFF_{p}$ \cite[Proposition 4]{BL1}. On appelle \emph{poids de Serre de $K$} une classe d'isomorphisme de telles repr\'esentations (ou, par abus, l'une des repr\'esentations $\sigma_{\vec{r}}$).

\subsection{Alg\`ebres de Hecke sph\'eriques de $GL_{2}(F)$}
\label{hecke}
Soient $H$ un sous-groupe ouvert, compact modulo $Z$ de $G$ et $\pi$ une repr\'esentation lisse de $H$. On leur associe une alg\`ebre de Hecke $\mathcal{H}_{G}(H,\sigma)$ d\'efinie par :
$$ \mathcal{H}_{G}(H, \sigma) := \text{End}_{G}(\text{c-ind}_{H}^{G}(\sigma)) \ .$$ 
Consid\'erons le cas o\`u $H = KZ$ et $\sigma = \sigma_{\vec{r}}$ est l'un des poids de Serre d\'efinis dans la Section \ref{SerreGL}. Barthel et Livn\'e \cite[Proposition 8]{BL1} ont alors isol\'e un op\'erateur $T_{\vec{r}} \in \mathcal{H}_{G}(KZ, \sigma_{\vec{r}})$ qui engendre $\mathcal{H}(KZ, \sigma_{\vec{r}})$ et fournit un isomorphisme de $\FFF_{p}$-alg\`ebres 
$$\mathcal{H}(KZ, \sigma_{\vec{r}}) \simeq \FFF_{p}[T_{\vec{r}}] \ .$$ 
Cet op\'erateur leur permet de d\'efinir les repr\'esentations conoyau suivantes : pour tout scalaire $\lambda \in \FFF_{p}$, on note $\pi(\vec{r}, \lambda, 1)$ le conoyau du morphisme $G$-\'equivariant 
$$ T_{\vec{r}} - \lambda : \text{c-ind}_{KZ}^{G}(\sigma_{\vec{r}}) \to  \text{c-ind}_{KZ}^{G}(\sigma_{\vec{r}})\ .$$
Ils d\'efinissent alors la notion de repr\'esentation supersinguli\`ere de $GL_{2}(F)$ comme suit \cite[page 290]{BL2} : une repr\'esentation lisse irr\'eductible $\pi$ de $GL_{2}(F)$ sur $\FFF_{p}$ est dite \emph{supersinguli\`ere} s'il existe un param\`etre $\vec{r}$ tel qu'\`a torsion par un caract\`ere lisse pr\`es, $\pi$ soit un quotient irr\'eductible de $\pi(\vec{r}, 0, 1)$.\\

D'autre part, on associe \`a tout couple $(g , v) \in G \times V_{\vec{r}}$ la fonction \'el\'ementaire $[g,v] \in ind_{K\*Z}^{G}(\sigma_{\vec{r}})$ d\'efinie par : 
$$ \forall \ h \in G, \  [g,v](h) := \left\{ \begin{array}{ll} 0 & \text{ si } h \not\in K\*Z\*g^{-1} \\ \sigma_{\vec{r}}(h\*g)(v) & \text{ si } h \in K\*Z\*g^{-1}\end{array} \right. \ .$$
On note $\overline{[g,v]}$ son image modulo $T_{\vec{r}}$ (i.e. dans $\pi(\vec{r},0,1)$). On a alors imm\'ediatement les deux propri\'et\'es suivantes : 
\begin{itemize}
\item[\tt $\mathit{i)}$] $\forall \ g$, $g_{1} \in G$, $\forall \ v \in \sigma_{\vec{r}}$, $g\*([g_{1},v]) = [g\*g_{1},v]$;
\item[\tt $\mathit{ii)}$] $ \forall \ g \in G$, $\forall \ k \in K\*Z$, $\forall \ v \in \sigma_{\vec{r}}$, $[g\*k, v] = [g, \sigma_{\vec{r}}(k)(v)]$. \\
\end{itemize}

\subsection{Donn\'ees immobili\`eres}
\label{immob}
On note $X$ l'arbre de Bruhat-Tits de $G$ (qui est celui de $SL_{2}$, cf. \cite{Ser} pour plus de d\'etails),  $x_{0}$ son << sommet standard >> (not\'e $v_{0}$ dans \cite{BL1}), qui correspond \`a la classe d'homoth\'etie du r\'eseau standard $\mathcal{O}_{F} \oplus \mathcal{O}_{F}$ du $F$-espace vectoriel $F^{2}$, et $x_{1} := \alpha\*x_{0}$.\\

L'arbre $X$ est naturellement muni d'une distance (la distance entre deux sommets est d\'efinie comme le nombre minimal d'ar\^etes constituant un chemin dans $X$ entre ces sommets) ainsi que d'une action transitive de $G$ pour laquelle $x_{0}$ admet $KZ$ comme stabilisateur. Ceci permet notamment d'identifier le support des fonctions $[g,v]$ aux sommets de $X$ (en identifiant le support de $[g, v]$ au sommet $gx_{0}$) et, de mani\`ere plus g\'en\'erale, le support de toute fonction $f \in \text{c-ind}_{K\*Z}^{G}(\sigma_{\vec{r}})$ \`a une famille finie de sommets de $X$.\\
On appellera \emph{cercle} (resp. \emph{boule}) \emph{de rayon $n$} l'ensemble des sommets de $X$ situ\'es \`a distance $n$ (resp. $\leq n$) du sommet $x_{0}$.\\

\subsection{Compl\'ements concernant $SL_{2}(F)$}
\subsubsection{Notations}
Nous nous int\'eressons dans ce papier au groupe sp\'ecial lin\'eaire $G_{S}:= SL_{2}(F)$ dont un sous-groupe compact maximal est donn\'e par $K_{S} := SL_{2}(\mathcal{O}_{F})$. Celui-ci contient le centre $Z_{S} := \{I_{2}, -I_{2}\}$ de $G_{S}$, qui est trivial lorsque $p = 2$. On note $I_{S} = I \cap K_{S}$ le sous-groupe d'Iwahori standard de $G_{S}$, $I_{S}(1) = I(1) \cap K_{S}$ son pro-$p$-Iwahori, $B_{S} = B \cap G_{S}$ le sous-groupe de Borel de $G_{S}$ constitu\'e des matrices triangulaires sup\'erieures, $T_{S}$ son tore maximal d\'eploy\'e (form\'e de l'ensemble des matrices diagonales), de sorte que l'on a de nouveau $B_{S} = T_{S} U$. On introduit aussi les matrices suivantes de $G_{S}$ : 
$$ s := \left( \begin{array}{cc} 0 & -1 \\ 1 & 0 \end{array}\right) \ ; \ \alpha_{0} := \left( \begin{array}{cc} \varpi_{F} & 0 \\ 0 & \varpi_{F}^{-1} \end{array}\right) \ ; \ \beta_{0} := \left( \begin{array}{cc} 0 & -\varpi_{F}^{-1} \\ \varpi_{F} & 0 \end{array}\right) \ .$$
En particulier, on a $\beta_{0} = s\alpha_{0}$ et $s = \alpha_{0}\beta_{0}$.\\

\subsubsection{Donn\'ees immobili\`eres}
Comme nous l'avons d\'ej\`a signal\'e, le groupe $G_{S}$ poss\`ede le m\^eme arbre de Bruhat-Tits que $G$, sur lequel $G_{S}$ agit cette fois de mani\`ere non transitive. Plus pr\'ecis\'ement, l'action de $G_{S}$ partitionne les sommets de $X$ en deux orbites (disjointes) : celle de $x_{0}$, que l'on notera $X_{p}$ et qui est constitu\'ee des cercles de rayon pair dans $X$, et celle de $x_{1}$, que l'on notera $X_{imp}$ et qui est constitu\'ee des cercles de rayon impair dans $X$. Le r\'esultat suivant est alors imm\'ediat.
\begin{prop}
Toute fonction $f \in \text{c-ind}_{K\*Z}^{G}(\sigma_{\vec{r}})$ se d\'ecompose de mani\`ere unique sous la forme 
$$f = f_{p} + f_{imp}$$
 avec $f_{p}$ (resp. $f_{imp}$) \'el\'ement de c-ind$_{K\*Z}^{G}(\sigma_{\vec{r}})$ \`a support dans $X_{p}$ (resp. $X_{imp}$).
\label{decoupe}
\end{prop}
\noindent De plus, les formules donnant l'action de $T_{\vec{r}}$ sur c-ind$_{K\*Z}^{G}(\sigma_{\vec{r}})$ \cite[(4) page 7]{Br1} montrent que l'action de $T_{\vec{r}}$ sur $X$, d\'efinie par la description du support de $T_{\vec{r}}(f)$ sur l'arbre en fonction du support de $f \in \text{c-ind}_{K\*Z}^{G}(\sigma_{\vec{r}})$, v\'erifie les deux inclusions suivantes : 
\begin{equation}
\left\{ \begin{array}{c} T_{\vec{r}}(X_{p}) \subset X_{imp} \ ; \\ T_{\vec{r}}(X_{imp}) \subset X_{p} \ . \end{array} \right. 
\label{echange}
\end{equation}\\ 

\subsubsection{Poids de Serre et supersingularit\'e}
\noindent De mani\`ere analogue \`a ce qui est fait dans le cas de $GL_{2}$, on appelle \emph{poids de Serre de $K_{S}$} une classe d'isomorphisme de repr\'esentations lisses irr\'eductibles de $K_{S}Z_{S}$ sur $\FFF_{p}$ (ou, par abus, une telle repr\'esentation). Ces poids de Serre sont bien connus \cite[Section 1]{Jey} : un syst\`eme de repr\'esentants en est fourni par les repr\'esentations inflat\'ees \`a partir de la restriction \`a $SL_{2}(k_{F})$ des repr\'esentations $\sigma_{\vec{r}}$ d\'efinies dans la section \ref{SerreGL}.\\
D'autre part, on conna\^it la structure des alg\`ebres de Hecke sph\'eriques associ\'ees \`a ces poids \cite[Th\'eor\`eme 1]{Moivieux}: elles sont isomorphes \`a des alg\`ebres de polyn\^omes $\FFF_{p}[\tau_{\vec{r}}]$, o\`u l'op\'erateur $\tau_{\vec{r}}$ agit sur les fonctions $[g, v]$ (et sur l'arbre $X$) comme le fait l'op\'erateur $T_{\vec{r}}^{2}$. Ceci nous permet de donner la d\'efinition suivante pour les repr\'esentations supersinguli\`eres de $G_{S}$ : ce sont, lorsqu'ils existent, les quotients irr\'eductibles des repr\'esentations conoyau $\displaystyle \frac{ \text{c-ind}_{K_{S}Z_{S}}^{G_{S}}(\sigma_{\vec{r}})}{\tau_{\vec{r}}( \text{c-ind}_{K_{S}Z_{S}}^{G_{S}}(\sigma_{\vec{r}}) )}$.\\

\section{Lien entre les repr\'esentations de $GL_{2}(F)$ et de $SL_{2}(F)$}
\label{restriction}
\begin{prop}
Supposons que $F$ soit de caract\'eristique diff\'erente de $2$.\\
Soit $\sigma$ une repr\'esentation lisse irr\'eductible \`a caract\`ere central de $G_{S}$. Il existe une repr\'esentation lisse irr\'eductible $\pi$ de $G$ ayant un caract\`ere central prolongeant celui de $\sigma$ et telle que la restriction de $\pi$ \`a $G_{S}$ contienne $\sigma$.
\label{restri}
\end{prop}
\begin{proof}
Plusieurs d\'emonstrations de ce r\'esultat sont possibles. Nous donnons ici une preuve utilisant un argument de filtration par le socle : soit $\sigma$ une repr\'esentation lisse irr\'eductible de $G_{S}$. On l'\'etend en une repr\'esentation lisse de $G_{S}Z$ (encore not\'ee $\sigma$) en faisant agir $Z$ par un caract\`ere lisse \'etendant l'action de $Z_{S}$, puis l'on consid\`ere la repr\'esentation induite $\Pi := Ind_{G_{S}Z}^{G}(\sigma)$. Comme $G_{S}Z$ est un sous-groupe normal d'indice fini de $G$, $\Pi$ est une repr\'esentation de $G$ de longueur finie, ce qui signifie qu'elle peut \^etre filtr\'ee par une suite finie de $G$-modules 
$$ \{0\} = \Pi_{m} \subset \Pi_{m-1} \subset \ldots \subset \Pi_{1} \subset \Pi_{0} = \Pi$$
telle que pour tout $i \in \{0, \ldots, m-1\}$, le quotient $\pi_{i} := \Pi_{i} / \Pi_{i+1}$ soit un $G$-module irr\'eductible.\\
Notons $r \in \{0, \ldots , m-1\}$ le plus grand entier tel que $\sigma \cap \Pi_{r}$ soit non nul. Par irr\'eductibilit\'e de $\sigma$ comme repr\'esentation de $G_{S}$, on obtient que $\sigma$ est contenue dans $\Pi_{r}$, tandis que la maximalit\'e de $r$ implique que $\sigma \cap \Pi_{r+1}$ est nul. Ceci permet de voir $\sigma$ comme une sous-repr\'esentation de $\Pi_{r} / \Pi_{r+1} = \pi_{r}$, qui est par construction une repr\'esentation de $G$ irr\'eductible. On peut donc conclure en prenant $\pi := \pi_{r}$.
\end{proof}

\begin{NB} On peut m\^eme d\'emontrer \cite[Lemma 2.4]{LL} que, lorsque $F$ est de caract\'eristique nulle (i.e. une extension finie de $\QQ_{p}$), la restriction de toute repr\'esentation lisse irr\'eductible admissible de $G$ \`a $G_{S}$ se d\'ecompose en une somme directe finie de repr\'esentations lisses irr\'eductibles de $G_{S}$. 
\end{NB}

\section{Repr\'esentations non supersinguli\`eres de $SL_{2}(F)$}
\label{classifnssg}
\subsection{Caract\`eres de $G_{S}$}
De mani\`ere g\'en\'erale, l'ab\'elianis\'e de $SL_{n}$ est trivial pour tout entier $n \geq 2$, de sorte que tout homomorphisme de groupes de $SL_{n}$ \`a valeurs dans un groupe ab\'elien est trivial. On en d\'eduit que l'on a en particulier l'\'enonc\'e suivant.
\begin{prop}
Le seul caract\`ere lisse de $G_{S}$ \`a valeurs dans $\FFF_{p}^{\times}$ est le caract\`ere trivial.
\label{caracgs}
\end{prop}

\subsection{Repr\'esentations de la s\'erie principale}
\label{sp}
\subsubsection{Terminologie}
On commence par rappeler la description des caract\`eres modulaires lisses du sous-groupe de Borel $B_{S}$. 
\begin{lm}
Tout caract\`ere lisse de $B_{S}$ \`a valeurs dans $\FFF_{p}^{\times}$ s'obtient par inflation \`a partir d'un caract\`ere lisse $\eta : F^{\times} \to \FFF_{p}^{\times}$. On note encore $\eta$ (ou $\eta \otimes \mathbf{1}$) ce caract\`ere.
\label{caracbs}
\end{lm}
\begin{proof}
Soit $\chi : B_{S} \to \FFF_{p}^{\times}$ un caract\`ere lisse. Comme ($\FFF_{p}^{\times}$, $\times$) est un groupe ab\'elien, $\chi$ se factorise \`a travers l'ab\'elianis\'e de $B_{S}$, qui n'est autre que le tore diagonal $T_{S} \simeq F^{\times}$. Autrement dit, il existe un caract\`ere lisse $\eta : F^{\times} \to \FFF_{p}^{\times}$ tel que : $\forall \ \lambda \in F^{\times}$, $\forall \ b \in F$,
$$ \chi(\left( \begin{array}{cc} \lambda & b \\ 0 & \lambda^{-1} \end{array}\right)) = \eta(\lambda)$$
ce qui est le r\'esultat demand\'e.
\end{proof}
\begin{NB}
La m\^eme construction permet d'obtenir un caract\`ere $\eta := \eta_{1} \otimes \eta_{2}$ de $B$ \`a partir de deux caract\`eres lisses $\eta_{1}, \eta_{2}$ de $F^{\times}$. Lorsque l'on restreint $\eta$ \`a $B_{S}$, on r\'ecup\`ere le caract\`ere $(\eta_{1}\eta_{2}^{-1})\otimes \mathbf{1}= \eta_{1}\eta_{2}^{-1}$.
\end{NB}

\noindent Soit $\eta : B_{S} \to \FFF_{p}^{\times}$ un caract\`ere lisse. On lui associe par induction lisse \`a partir de $B_{S}$ une repr\'esentation de $G_{S}$ not\'ee  $Ind_{B_{S}}^{G_{S}}(\eta)$ qui a pour espace de repr\'esentation le $\FFF_{p}$-espace vectoriel des fonctions $f : G_{S} \to \FFF_{p}$ localement constantes telles que 
$$ \forall (b,g) \in B_{S} \times G_{S}, \  f(bg) = \eta(b) f(g)\} $$
sur lequel $G_{S}$ agit par translations \`a droite : 
$$ \forall \ (g, x) \in G_{S} \times G_{S}, \ (g \cdot f)(x) := f(xg) \ .$$
Lorsqu'une telle repr\'esentation est irr\'eductible, on dit que c'est une repr\'esentation \emph{de la s\'erie principale} de $G_{S}$. La m\^eme terminologie vaut pour les repr\'esentations de $G$ obtenues par induction parabolique \`a partir d'un caract\`ere lisse modulo $p$ de $B$.\\

\subsubsection{Crit\`ere d'irr\'eductibilit\'e des s\'eries principales}
\label{irredspal}
L'objectif de cette partie est de d\'emontrer le th\'eor\`eme suivant.
\begin{thm}
Soit $\eta : B_{S} \to \FFF_{p}^{\times}$ un caract\`ere lisse non trivial. La repr\'esentation $Ind_{B_{S}}^{G_{S}}(\eta)$ est alors une repr\'esentation lisse irr\'eductible de $G_{S}$ ayant pour caract\`ere central $\eta$.
\label{sppal}
\end{thm}

Le seul point non trivial de cet \'enonc\'e est l'assertion d'irr\'eductibilit\'e. Une premi\`ere m\'ethode de d\'emonstration consiste \`a reprendre les \'etapes de la preuve effectu\'ee par Barthel-Livn\'e pour les s\'eries principales de $GL_{2}(F)$ et \`a les adapter au cas de $SL_{2}(F)$ : c'est ce que nous faisons dans l'Annexe \ref{spBL}.\\
Une seconde m\'ethode possible repose sur l'\'etude de la restriction \`a $B_{S}$ des s\'eries principales de $G_{S}$. Un premier int\'er\^et est qu'elle semble plus susceptible que la premi\`ere d'\^etre g\'en\'eralisable, comme en attestent par exemple les travaux de Vign\'eras \cite{Vig08}. Un second avantage est qu'elle ne n\'ecessite pas de distinguer le cas ramifi\'e du cas non ramifi\'e, ce qui est le cas dans la m\'ethode de Barthel-Livn\'e.\\

Nous donnons ici les grandes \'etapes de la seconde m\'ethode : pour les d\'etails, nous renvoyons \`a l'Annexe \ref{sppaljac}.
On commence par observer que l'application d'\'evaluation en $I_{2}$ permet facilement de voir\footnote{On peut aussi voir ce r\'esultat comme un cas particulier de \cite[Th\'eor\`eme 1]{Vig08}.} que la restriction \`a $B_{S}$ de $Ind_{B_{S}}^{G_{S}}(\eta)$ est de longueur $2$ avec pour quotient le caract\`ere $\eta$ et pour sous-objet le sous-espace $W$ des \'el\'ements \`a support dans $B_{S}sU$. Par suite, si $Ind_{B_{S}}^{G_{S}}(\eta)$ est une repr\'esentation r\'eductible de $G_{S}$, elle est elle aussi de longueur 2 et admet un sous-quotient de dimension 1, qui est n\'ecessairement le caract\`ere trivial de $G_{S}$ par la Proposition \ref{caracgs}. L'\'etude du module de Jacquet associ\'e \`a $Ind_{B_{S}}^{G_{S}}(\eta)$ implique alors, par un raisonnement analogue \`a celui men\'e dans le cadre classique des repr\'esentations complexes, que $\eta$ doit \^etre le caract\`ere trivial de $B_{S}$, ce qui prouve le Th\'eor\`eme \ref{sppal} par contraposition.

\subsection{Repr\'esentations en s\'eries sp\'eciales}
Pour compl\'eter ce qui a \'et\'e fait dans la section pr\'ec\'edente, il nous reste \`a \'etudier la repr\'esentation $Ind_{B_{S}}^{G_{S}}(\mathbf{1})$ obtenue par induction parabolique du caract\`ere trivial de $B_{S}$. Il est facile de voir qu'elle n'est pas irr\'eductible puisque l'espace des fonctions constantes sur $G_{S}$, sur lequel $G_{S}$ agit trivialement, en est une sous-repr\'esentation propre \'evidente. Notons $\mathbf{1}$ cette repr\'esentation et consid\'erons le quotient $Sp_{S} := Ind_{B_{S}}^{G_{S}}(\mathbf{1}) / \mathbf{1}$ qui lui est associ\'e. Nous allons d\'emontrer le r\'esultat suivant, qui est le pendant de \cite[Theorem 28]{BL2}.
\begin{thm}
La repr\'esentation $Sp_{S}$ est irr\'eductible.
\label{irredsp}
\end{thm}
\noindent Pour ce faire, nous allons donner une autre description de $Sp_{S}$ qui nous permettra de conclure par l'\'etude de l'espace des $I_{S}(1)$-invariants de $Sp_{S}$.\\
La d\'ecomposition de Bruhat pour $G_{S}$ implique que le quotient $B_{S} \backslash G_{S}$ est topologiquement isomorphe \`a la droite projective $\PP^{1}(F)$. Par suite, la repr\'esentation $Ind_{B_{S}}^{G_{S}}(1)$ s'identifie \`a l'espace $\mathcal{S}(\PP^{1}(F))$ des fonctions localement constantes sur $\PP^{1}(F)$ \`a valeurs dans $\FFF_{p}$ sur lequel $G_{S}$ agit par translations \`a droite, et $Sp_{S}$ s'identifie au quotient $\mathcal{S}(\PP^{1}(F))/\{\text{constantes}\} \simeq \mathcal{S}(\PP^{1}(F))/ W$ avec $W$ un $\FFF_{p}$-espace vectoriel de dimension 1 sur lequel $I_{S}(1)$ agit trivialement.\\ 

\noindent Les preuves des Lemmes 26 et 27 de \cite{BL2} sont encore valables pour $G_{S}$ et $I_{S}$, ce qui fournit directement l'\'enonc\'e suivant.

\begin{lm}
On dispose de la suite exacte courte suivante de $I_{S}(1)$-modules : 
$$0 \longrightarrow W  \longrightarrow (\mathcal{S}(\PP^{1}(F)))^{I_{S}(1)} \longrightarrow (Sp_{S})^{I_{S}(1)} \longrightarrow 0  \ .$$
\label{suitex}
\end{lm}
\begin{NB}
La d\'ecomposition $G_{S} = B_{S}I_{S}(1) \sqcup B_{S}\beta_{0}I_{S}(1)$ permet de nouveau de d\'emontrer que le $G_{S}$-module $\mathcal{S}(\PP^{1}(F))$ est engendr\'e par le sous-espace de dimension 2 form\'e par ses vecteurs $I_{S}(1)$-invariants.
\label{invrk}
\end{NB}

\noindent \textbf{Preuve du Th\'eor\`eme \ref{irredsp} :}
Soit $\pi$ une sous-repr\'esentation non nulle de $Sp_{S}$. La Proposition \ref{faitcrucial} appliqu\'ee au pro-$p$-groupe $I_{S}(1)$ assure que l'on a $\dim_{\FFF_{p}} \pi^{I_{S}(1)} \geq 1$. Notons $V$ le rel\`evement de $\pi$ \`a $\mathcal{S}(\PP^{1}(F))$. Le Lemme \ref{suitex} assure alors que l'on a la cha\^ine d'in\'egalit\'es suivante : 
$$ 2 = \dim_{\FFF_{p}} (\mathcal{S}(\PP^{1}(F)))^{I_{S}(1)} \geq \dim_{\FFF_{p}} V^{I_{S}(1)} = 1 + \dim_{\FFF_{p}} \pi^{I_{S}(1)} \geq 2 \ ,$$
qui implique que $\dim_{\FFF_{p}} V^{I_{S}(1)} = 2$, donc que $V$ poss\`ede les m\^emes $I_{S}(1)$-invariants que $\mathcal{S}(\PP^{1}(F))$. La Remarque \ref{invrk} implique alors que $V = \mathcal{S}(\PP^{1}(F))$ et donc que $\pi = Sp_{S}$, ce qui prouve le Th\'eor\`eme \ref{irredsp}.\\

Pour terminer cette section, signalons que nous disposons du r\'esultat suivant, qui est l'analogue de \cite[Theorem 29]{BL2} pour $Sp_{S}$.
\begin{thm}
La repr\'esentation $Sp_{S}$ est le seul sous-quotient non trivial de $Ind_{B_{S}}^{G_{S}}(\mathbf{1})$.
\label{seulquo}
\end{thm}
\begin{proof}
On suit de nouveau le raisonnement de Barthel-Livn\'e : comme $Ind_{B_{S}}^{G_{S}}(\mathbf{1})$ est une repr\'esentation de $G_{S}$ de longueur 2, il nous suffit de prouver que la suite exacte 
$$ 0 \longrightarrow \mathbf{1} \longrightarrow Ind_{B_{S}}^{G_{S}}(\mathbf{1}) \longrightarrow Sp_{S} \longrightarrow 0$$
est non scind\'ee pour conclure.\\
Si $\ell$ est une fonction $G_{S}$-invariante sur $\mathcal{S}(\PP^{1}(F))$, elle v\'erifie  
$$ \begin{array}{rcll} \ell(\mathbf{1}_{\mathcal{O}_{0}}) & = & \ell(\mathbf{1}_{\mathcal{O}_{F}}) & \\
& = & \displaystyle \sum_{x \in k_{F}} \ell(\mathbf{1}_{x + \varpi_{F}\mathcal{O}_{F}}) & \\
& = & \displaystyle \sum_{x \in k_{F}} \ell(\left( \begin{array}{cc} \varpi_{F} & x \\ 0 & \varpi_{F}^{-1}\end{array}\right) \mathbf{1}_{\mathcal{O}_{F}}) & \\
& = & \displaystyle q  \ell(\mathbf{1}_{\mathcal{O}_{F}}) & \text{ par $G_{S}$-invariance de }\ell \\
& = & 0 & \text{ car $q =$ card $k_{F}$ est nul modulo $p$.}
\end{array}$$
Comme on a aussi
$$ \ell(\mathbf{1}_{\mathcal{O}_{\infty}}) = \ell(\beta_{0}\mathbf{1}_{{O}_{0}}) = \ell(\mathbf{1}_{\mathcal{O}_{0}}) = 0 \ ,$$
on obtient que $\ell$ est identiquement nulle puisque $\mathbf{1}_{\mathcal{O}_{\infty}}$ et $\mathbf{1}_{\mathcal{O}_{0}}$ engendrent $\mathcal{S}(\PP^{1}(F))$.
\end{proof}

\subsection{Preuve du Th\'eor\`eme \ref{dejairred}}
On rappelle tout d'abord le Th\'eor\`eme suivant d\^u \`a Barthel et Livn\'e \cite[Theorem 30]{BL1}.\footnote{Avec les notations de  \cite[Theorem 30]{BL1}, on a ici $\eta = \chi_{1}\chi_{2}^{-1}$. On peut se limiter \`a consid\'erer ce cas, quitte \`a tordre par le caract\`ere $\chi_{2} \circ \det$ (ce qui ne change pas les propri\'et\'es d'irr\'eductibilit\'e).}
\begin{thm}
Soit $\eta : F^{\times} \to \FFF_{p}^{\times}$ un caract\`ere lisse.
\begin{enumerate}
\item Si $\eta$ est non trivial, alors $Ind_{B}^{G}(\eta \otimes \mathbf{1})$ est une repr\'esentation irr\'eductible de $G$.
\item La repr\'esentation $Ind_{B}^{G}(\mathbf{1})$ est de longueur 2; elle contient comme unique sous-objet la repr\'esentation triviale, et son quotient est isomorphe \`a la repr\'esentation sp\'eciale $Sp$.
\end{enumerate}
\label{BLbis}
\end{thm}

\noindent Pour d\'emontrer le Th\'eor\`eme \ref{dejairred}, on remarque tout d'abord qu'il existe un plongement naturel des induites paraboliques de $G$ dans celles de $G_{S}$.
\begin{prop}
Soit $\eta : F^{\times} \to \FFF_{p}^{\times}$ un caract\`ere lisse. L'application naturelle de restriction induit un plongement $G_{S}$-\'equivariant 
$$ Ind_{B}^{G}(\eta \otimes \mathbf{1}) \hookrightarrow Ind_{B_{S}}^{G_{S}}(\eta) \ .$$
\label{propcleG}
\end{prop}
\begin{proof}
Soit $f \in Ind_{B}^{G}(\eta \otimes  \mathbf{1})$. Sa restriction \`a $G_{S}$ est alors une fonction lisse qui satisfait :
$$ \forall \ b \in B_{S}, \ \forall g \in G_{S}, \ f(bg) = \eta(b) f(g) \ , $$
de sorte qu'elle d\'efinit bien un \'el\'ement de $Ind_{B_{S}}^{G_{S}}(\eta)$. De plus, l'application de restriction est clairement $G_{S}$-\'equivariante puisque $G_{S}$ agit sur les deux espaces de la m\^eme mani\`ere. Reste donc \`a prouver son injectivit\'e, qui provient de la d\'ecomposition $G = BG_{S}$. En effet, si $f \in Ind_{B}^{G}(\eta)$ est de restriction \`a $G_{S}$ identiquement nulle et si $g \in G$ se d\'ecompose sous la forme
$$ g = \left( \begin{array}{cc} u & 0 \\ 0 & 1\end{array}\right) g_{s} $$ 
avec $u := \det g \in F^{\times}$ et $g_{s} \in G_{S}$, on a alors 
$$ f(g) = f(ug_{s}) = \eta(u) f(g_{s}) = 0 \ ,$$
ce qui montre que $f$ est identiquement nulle et termine la d\'emonstration.
\end{proof}

\begin{cor}
Le premier point du Th\'eor\`eme \ref{dejairred} est vrai.
\end{cor}
\begin{proof}
La repr\'esentation $Ind_{B}^{G}(\eta)$ est une repr\'esentation non nulle de $G_{S}$ et, sous les hypoth\`eses demand\'ees, $Ind_{B_{S}}^{G_{S}}(\eta)$ est une repr\'esentation irr\'eductible de $G_{S}$. Le plongement de la Proposition \ref{propcleG} est donc n\'ecessairement un isomorphisme de $G_{S}$-repr\'esentations. 
\end{proof}

Nous d\'emontrons maintenant les deux derniers points du Th\'eor\`eme \ref{dejairred}. D'apr\`es le second point du Th\'eor\`eme \ref{BLbis}, la restriction \`a $G_{S}$ de $Ind_{B}^{G}(\mathbf{1})$ contient la repr\'esentation triviale de $G_{S}$. Du coup, la Proposition \ref{propcleG} fournit naturellement un morphisme $G_{S}$-\'equivariant non nul
\begin{equation}
Ind_{B}^{G}(\mathbf{1}) / \mathbf{1} \to Ind_{B_{S}}^{G_{S}}(\mathbf{1}) / \mathbf{1} \ .
\label{morphsp}
\end{equation}
Le membre de gauche de ce morphisme est la repr\'esentation sp\'eciale $Sp$ de $G$, dont la restriction \`a $G_{S}$ est non nulle \cite[Preuve de la Proposition 32]{BL2}, tandis que le membre de droite est la repr\'esentation sp\'eciale $Sp_{S}$ de $G_{S}$ dont on a prouv\'e l'irr\'eductibilit\'e au Th\'eor\`eme \ref{irredsp}. Ceci implique que le morphisme (\ref{morphsp}) est un isomorphisme de $G_{S}$-repr\'esentations, ce qui prouve le second point du Th\'eor\`eme \ref{dejairred}. Le troisi\`eme point en s'en d\'eduit alors directement par application du Lemme des cinq au diagramme \`a lignes exactes suivant : 
  $$ \xymatrix{
  0\ar[r] & \mathbf{1} \ar[r] \ar[d]^{\simeq} &  Ind_{B}^{G}( \mathbf{1})\ar[r] \ar[d] &  Sp\ar[r] \ar[d]^{\simeq}& 0\\
  0\ar[r] &  \mathbf{1} \ar[r] & Ind_{B_{S}}^{G_{S}}( \mathbf{1}) \ar[r] &  Sp_{S}\ar[r] & 0
  } $$
 
 \subsection{Quelques cons\'equences utiles}
Nous donnons ici quelques cons\'equences utiles du Th\'eor\`eme \ref{dejairred} qui permettent notamment d'obtenir ce qui nous manque pour prouver le Th\'eor\`eme \ref{nonssg}. Avant cela, on rappelle un fait simple mais important qui d\'ecoule directement de la d\'efinition des repr\'esentations supersinguli\`eres : \emph{toute repr\'esentation de $G_{S}$ isomorphe \`a une repr\'esentation supersinguli\`ere est elle-m\^eme supersinguli\`ere}.
 \begin{cor}
 Les repr\'esentations de la s\'erie principale de $G_{S}$ et la repr\'esentation $Sp_{S}$ ne sont pas supersinguli\`eres.
 \label{classnonssg}
 \end{cor}
 \begin{proof}
 La preuve de ce r\'esultat repose sur le Th\'eor\`eme 30 (3) de \cite{BL1} et sur le fait que les actions de $T_{\vec{r}}^{2} \in \mathcal{H}_{G}(KZ, \sigma_{\vec{r}})$ et de $\tau_{\vec{r}} \in \mathcal{H}_{G_{S}}(K_{S}Z_{S}, \sigma_{\vec{r}})$ sur l'arbre de Bruhat-Tits $X$ co\"incident.\\
Nous traitons le cas des repr\'esentations de la s\'erie principale, le cas de la repr\'esentation sp\'eciale s'obtenant de mani\`ere analogue. Soit donc $\eta$ un caract\`ere lisse non trivial de $F^{\times}$, que l'on \'ecrit sous la forme $\eta := \mu_{\lambda}\omega^{\vec{r}}$ avec $\lambda \in \FFF_{p}^{\times}$ et $\vec{r} \in \{0 , \ldots, p-1\}^{f}$. 
D'apr\`es \cite[Theorem 30 (3)]{BL1}, on sait que l'on dispose alors d'un isomorphisme de repr\'esentations de $G$ : 
$$ Ind_{B}^{G}(\eta \otimes \mathbf{1}) \simeq \pi(r, \lambda^{-1/2}, \eta\mu_{\lambda^{-1/2}})$$
o\`u $\lambda^{-1/2} \in \FFF_{p}^{\times}$ est une racine carr\'ee de $\lambda^{-1}$. Le Th\'eor\`eme \ref{dejairred} assure que cet isomorphisme induit l'isomorphisme de repr\'esentations de $G_{S}$ suivant :
\begin{equation}
Ind_{B_{S}}^{G_{S}}(\eta) \simeq \pi(\vec{r} ,  \lambda^{-1/2}, 1) \vert_{G_{S}} \ .
\label{isoGS}
\end{equation}
D'apr\`es  \cite[Theorem 19]{BL1}, le membre de gauche de (\ref{isoGS}) est muni d'une structure de $\mathcal{H}_{G}(KZ, \sigma_{\vec{r}})$-module \`a partir de l'action de cette alg\`ebre de Hecke sur l'arbre de Bruhat-Tits $X$. En particulier, on sait que l'op\'erateur $T_{\vec{r}}$ agit sur $\pi(\vec{r}, \lambda^{-1/2}, 1)$ par le scalaire $\lambda^{-1/2}$, donc que l'op\'erateur $T_{\vec{r}}^{2}$ agit dessus par le scalaire $\lambda^{-1} \not= 0$. Ceci implique que l'op\'erateur $\tau_{\vec{r}}$, qui agit sur $X$ de la m\^eme mani\`ere que $T_{\vec{r}}^{2}$, ne peut agir par $0$ sur $\pi(\vec{r} , \lambda^{-1/2}, 1) \vert_{G_{S}}$, et donc que $Ind_{B_{S}}^{G_{S}}(\eta)$ n'est pas une repr\'esentation supersinguli\`ere de $G_{S}$. 
\end{proof}

\begin{cor}
Soient $\sigma$ une repr\'esentation lisse irr\'eductible de $G_{S}$ et $\Sigma$ une repr\'esentation lisse irr\'eductible de $G$ telles que $\sigma$ soit contenue dans $\Sigma \vert_{G_{S}}$. Alors $\Sigma$ est supersinguli\`ere pour $G$ si et seulement si $\sigma$ est supersinguli\`ere pour $G_{S}$.
 \label{respessg1}
 \end{cor}
 \begin{proof}
 Nous allons d\'emontrer cette double implication par contraposition. Supposons tout d'abord que $\sigma$ ne soit pas supersinguli\`ere pour $G_{S}$ : le Th\'eor\`eme \ref{dejairred} implique alors que $\sigma$ est isomorphe \`a la restriction \`a $G_{S}$ d'une repr\'esentation irr\'eductible non supersinguli\`ere de $G$; elle ne peut donc pas \^etre contenue dans une repr\'esentation irr\'eductible supersinguli\`ere de $G$ \`a cause de \cite[Corollary 36]{BL1}, qui affirme qu'il n'existe pas d'entrelacements entre repr\'esentations supersinguli\`eres et non supersinguli\`eres de $G$.\\
R\'eciproquement, supposons que $\Sigma$ ne soit pas une repr\'esentation supersinguli\`ere de $G$ : la classification \'etablie dans \cite[page 290]{BL1} assure alors que $\Sigma$ est isomorphe \`a une repr\'esentation de la s\'erie principale de $G$ ou \`a une repr\'esentation sp\'eciale de $G$. De nouveau par le Th\'eor\`eme \ref{dejairred} (et par irr\'eductibilit\'e de $\sigma$), on obtient que $\sigma$ est isomorphe \`a une repr\'esentation de la s\'erie principale de $G_{S}$ ou \`a la repr\'esentation $Sp_{S}$. Le Corollaire \ref{classnonssg} implique alors que $\sigma$ ne peut pas \^etre supersinguli\`ere.
\end{proof}

\subsection{Preuve du Th\'eor\`eme  \ref{nonssg}}
On suppose dans cette partie que $F$ est de caract\'eristique diff\'erente de 2 afin de pouvoir utiliser la Proposition \ref{restri}. En la combinant avec \cite[Theorem 33]{BL1}, on obtient tout d'abord que les quatre familles mentionn\'ees dans l'\'enonc\'e du Th\'eor\`eme \ref{nonssg} \'epuisent \`a isomorphisme pr\`es toutes les repr\'esentations lisses irr\'eductibles admissibles de $G_{S}$ sur $\FFF_{p}$.\\
Le fait que ces familles soient disjointes d\'ecoule directement de ce que l'on a montr\'e dans les sections pr\'ec\'edentes. On commence par remarquer que le caract\`ere trivial est la seule repr\'esentation de dimension finie de $G_{S}$ qui appara\^it dans la liste, donc qu'elle ne peut \^etre isomorphe \`a aucune autre repr\'esentation de cette liste. Par ailleurs, le Th\'eor\`eme \ref{dejairred} associ\'e \`a \cite[Corollary 36 (1)]{BL1} assure qu'il n'existe pas d'entrelacement entre les trois premi\`eres familles, et le Corollaire \ref{classnonssg} implique que les repr\'esentations supersinguli\`eres ne peuvent \^etre isomorphes aux autres repr\'esentations de la liste, ce qui ach\`eve la d\'emonstration.

\begin{rem}
Comme nous l'avons remarqu\'e dans l'introduction, ceci implique que les repr\'esentations supersinguli\`eres de $SL_{2}(F)$ sont exactement ses repr\'esentations supercuspidales (i.e. qui ne sont pas sous-quotient d'une repr\'esentation en s\'erie principale).
\end{rem}  
\section{Repr\'esentations supersinguli\`eres de $SL_{2}(\QQ_{p})$}
\label{cassg}
\noindent Nous supposons d\'esormais que $F = \QQ_{p}$. Par suite, $k_{F}$ est le corps $\FF_{p}$ \`a $p$ \'el\'ements et l'on peut choisir (comme nous le faisons \`a partir de maintenant) $\varpi_{F} = p$ comme uniformisante. Puisque $f = 1$, nous all\`egerons les notations en notant $r$ au lieu de $\vec{r}$.\\ 
Cette section s'organise comme suit : nous commen\c{c}ons par rappeler quelques r\'esultats de Breuil concernant les repr\'esentations supersinguli\`eres de $GL_{2}(\QQ_{p})$ (section \ref{Breuilli}), que nous combinons ensuite aux r\'esultats de la Section \ref{restriction} pour en d\'eduire le Th\'eor\`eme \ref{ssgqp}.
\begin{rem}
Les r\'esultats inh\'erents aux non-isomorphismes entre repr\'esentations supersinguli\`eres de $G_{S}$ (Section \ref{isomgs}) peuvent \^etre obtenus bien plus rapidement \`a partir de l'\'etude des modules simples sur la pro-$p$-alg\`ebre de Hecke de $G_{S}$, comme cela est fait dans \cite{MoiRachelbis}.
\end{rem}

\subsection{Repr\'esentations supersinguli\`eres de $GL_{2}(\QQ_{p})$}
\label{Breuilli}
L'\'etude pr\'ecise des repr\'esentations supersinguli\`eres de $GL_{2}(\QQ_{p})$ est essentiellement due \`a Breuil \cite{Br1} qui a d\'emontr\'e, \`a l'aide des r\'esultats de Barthel-Livn\'e \cite[Theorem 34]{BL1}, le th\'eor\`eme suivant.
\begin{thm}
Les repr\'esentations supersinguli\`eres de $GL_{2}(\QQ_{p})$ sur $\FFF_{p}$ sont exactement les repr\'esentations $\pi(r,0,1)$ avec $r \in \{0, \ldots, p-1\}$.
\end{thm}
L'obtention de ce r\'esultat repose fondamentalement sur l'\'etude de la structure de l'espace des $I(1)$-invariants de $\pi(r,0,1)$ \cite[Th\'eor\`eme 3.2.4 \footnote{Nous donnons une forme l\'eg\`erement diff\'erente de celle de \cite{Br1},  mais un calcul direct \`a l'aide de la $K$-\'equivariance montre que l'on a $[\alpha, y^{r}] = [\beta, x^{r}]$ de sorte que nous avons effectivement le m\^eme \'enonc\'e.} et Remarque 3.2.5]{Br1}.
\begin{thm}
Pour tout $0 \leq r \leq p-1$, on a : 
\begin{equation}
 \pi(r,0,1)^{I(1)} = \FFF_{p} \overline{[I_{2}, x^{r}]} \oplus \FFF_{p}\overline{[\beta, x^{r}]} \ .
\end{equation} 
\label{Breuil} 
\end{thm}
L'\'etude de la preuve de ce th\'eor\`eme montre qu'elle reste valable si l'on cherche \`a calculer les $I_{S}(1)$-invariants de $\pi(r,0,1)$ : en effet, la d\'emonstration donn\'ee par Breuil prouve que l'espace des invariants de $\pi(r,0,1)$ sous le sous-groupe de $G$ engendr\'e par les matrices unipotentes inf\'erieures et sup\'erieures est \'egal au $\FFF_{p}$-espace vectoriel de dimension 2 ayant pour base $\{\overline{[I_{2}, x^{r}]}, \overline{[\beta, x^{r}]}\}$. Ces deux \'el\'ements sont $I_{S}(1)$-invariants (puisque $I(1)$-invariants), ce qui fournit un r\'esultat analogue au Th\'eor\`eme \ref{Breuil} pour $G_{S}$.
\begin{thm} Pour tout $r \in \{0, \ldots, p-1\}$, l'espace des vecteurs $I_{S}(1)$-invariants de $\pi(r,0,1)$ est \'egal \`a 
$$ \pi(r,0,1)^{I_{S}(1)} = \FFF_{p} \overline{[I_{2}, x^{r}]} \oplus  \FFF_{p} \overline{[\beta, x^{r}]} = \pi(r,0,1)^{I(1)} \ .$$
\label{BreuilMoi}
\end{thm}

\noindent Cet \'enonc\'e fait naturellement appara\^itre deux repr\'esentations de $G_{S}$ contenues dans $\pi(r,0,1)$ : la repr\'esentation engendr\'ee sur $\FFF_{p}$ par le vecteur $v_{r,\infty} := \overline{[I_{2}, x^{r}]}$, que nous noterons $\pi_{r,\infty}$, et celle engendr\'ee par $v_{r,0} := \overline{[\beta, x^{r}]}$, que nous noterons $\pi_{r,0}$.
\begin{NB}
\label{P1}
Si l'on veut \^etre (tr\`es) prudent, l'\'enonc\'e du Th\'eor\`eme \ref{BreuilMoi} pousse d'abord \`a consid\'erer toute la famille de vecteurs $I_{S}(1)$-invariants $v_{r,\lambda} := v_{r,0} + \lambda v_{r,\infty}$ ($\lambda \in \FFF_{p}$) et les sous-repr\'esentations $\pi_{r,\lambda}$ de $G_{S}$ qu'ils engendrent sur $\FFF_{p}$. Comme nous le verrons dans le Corollaire \ref{slt} de la section \ref{decompo}, les seules repr\'esentations irr\'eductibles qui apparaissent dans cette famille sont $\pi_{r,0}$ et $\pi_{r,\infty}$. Les autres sont en fait toutes \'egales \`a $\pi(r,0,1)$.
\label{piz}
\end{NB}

\subsection{Irr\'eductibilit\'e de $\pi_{r, \infty}$ et de $\pi_{r,0}$}
\label{irred}
Jusqu'\`a la fin de la section \ref{decompo}, on fixe un entier $r \in \{0, \ldots, p-1\}$. Comme nous allons le d\'emontrer dans la Section \ref{superired}, l'irr\'eductibilit\'e des repr\'esentations $\pi_{r,\infty}$ et $\pi_{r,0}$ repose essentiellement sur le fait que leurs espaces de vecteurs $I_{S}(1)$-invariants sont de dimension 1 (Proposition \ref{dim1}).

\subsubsection{Un r\'esultat technique important}
\begin{lm}
Les repr\'esentations $\pi_{r,0}$ et $\pi_{r,\infty}$ n'ont pas d'\'el\'ement commun non nul : 
$$\pi_{r,0} \cap \pi_{r,\infty} = \{0\} \ .$$
\label{rescle}
\end{lm}
\begin{proof}
Soient $\overline{f} \in \pi_{r, 0} \cap \pi_{r, \infty}$ et $f \in ind_{K\*Z}^{G}(\sigma_{r})$ un rel\`evement de $\overline{f}$ . Dire que $\overline{f}$ est dans $\pi_{r, \infty}$ signifie que l'on peut \'ecrire $f$ sous la forme d'une somme finie 
\begin{equation}
f = \displaystyle \sum_{i \in I} \lambda_{i}\*s_{i}\*v_{r,\infty} + T_{r}(f_{\infty})
\label{cle1}
\end{equation}
avec $\lambda_{i} \in \FFF_{p}$ et $s_{i} \in G_{S}$ pour tout $i \in I$, et $f_{\infty} \in ind_{K\*Z}^{G}(\sigma_{r})$.\\
De m\^eme, dire que $\overline{f}$ est dans $\pi_{r, 0}$ signifie que l'on peut l'\'ecrire sous la forme d'une somme finie 
\begin{equation}
f = \displaystyle \sum_{j \in J} \mu_{j}\*t_{j}\*v_{r,0} + T_{r}(f_{0})
\label{cle2}
\end{equation}
avec $\mu_{j} \in \FFF_{p}$ et $t_{j} \in G_{S}$ pour tout $j \in J$, et $f_{0} \in ind_{K\*Z}^{G}(\sigma_{r})$.\\
\noindent La Proposition \ref{decoupe} nous permet par ailleurs d'\'ecrire $f = f_{p} + f_{imp}$ (resp. $f_{\infty} = f_{\infty, p} + f_{\infty, imp}$, $ f_{0} = f_{0, p} + f_{0, imp}$) avec $f_{p}$ (resp. $f_{\infty, p}$, $f_{0,p}$) \`a support dans $X_{p}$ et $f_{imp}$ (resp. $f_{\infty, imp}$, $f_{0, imp}$) \`a support dans $X_{imp}$. On peut donc r\'e\'ecrire l'expression (\ref{cle1}) sous la forme
\begin{equation}
f_{imp} - T_{r}(f_{\infty, p}) = \displaystyle \sum_{i \in I} \lambda_{i}\*s_{i}\*v_{r,\infty} + T_{r}(f_{\infty, imp}) - f_{p} \ .
\label{cle1bis}
\end{equation}
Rappelons alors que $v_{r,\infty} = \overline{[I_{2}, x^{r}]}$ est \`a support dans le sommet $x_{0}$ et que l'on a $T_{r}(X_{imp}) \subset X_{p}$ par (\ref{echange}), de sorte que le membre de droite de (\ref{cle1bis}) est \`a support dans $X_{p}$. Le membre de gauche de (\ref{cle1bis}) est quant \`a lui de support contenu dans $X_{imp}$ car  (\ref{echange}) assure que $T_{r}(X_{p}) \subset X_{imp}$. Etant donn\'e que $X_{p} \cap X_{imp} = \emptyset$, l'expression (\ref{cle1bis}) n'est possible que si ses deux membres sont nuls, ce qui implique en particulier que $f_{imp} \in Im\ T_{r}$.\\

\noindent De m\^eme, l'expression (\ref{cle2}) peut \^etre r\'e\'ecrite sous la forme 
\begin{equation}
f_{imp} - T_{r}(f_{0, p}) = \displaystyle \sum_{j \in J} \mu_{j}\*t_{j}\*v_{r,0} + T_{r}(f_{0, imp}) - f_{p} \ .
\label{cle2bis}
\end{equation}
Sachant cette fois que $v_{r,0} = \overline{[\alpha , y^{r}]}$ est \`a support dans le sommet $x_{1}$, on obtient par le m\^eme raisonnement que le membre de droite de (\ref{cle2bis}) est \`a support dans $X_{imp}$ tandis que le membre de gauche est \`a support dans $X_{p}$, ce qui n'est de nouveau possible que si ces deux membres sont nuls. Ceci implique en particulier que $f_{imp} \in Im\ T_{r}$.\\
Par lin\'earit\'e de l'op\'erateur $T_{r}$, on obtient alors que $f = f_{p} + f_{imp} \in Im\ T_{r}$, ce qui signifie que $\overline{f}$ est nulle dans $\pi(r,0,1)$ et termine la d\'emonstration.
\end{proof}

\subsubsection{$I_{S}(1)$-invariants de $\pi_{r,0}$ et $ \pi_{r,\infty}$}
\begin{prop} Soit $ r \in \{0, \ldots, p-1\}$.\\
L'espace des $I_{S}(1)$-invariants de $\pi_{r,\infty}$ (resp. $\pi_{r,0}$) est un $\FFF_{p}$-espace vectoriel de dimension 1 engendr\'e par $v_{r,\infty}$ (resp. $v_{r,0}$).
\label{dim1}
\end{prop}
\begin{proof}
On traite le cas de $\pi_{r, \infty}$, celui de $\pi_{r,0}$ s'obtenant de la m\^eme mani\`ere. Il est tout d'abord clair que le Th\'eor\`eme \ref{BreuilMoi} assure que $\FFF_{p} v_{\infty}$ est contenu dans $\pi_{r,\infty}^{I_{S}(1)}$. R\'eciproquement, soit $f$ un \'el\'ement non nul de $\pi_{r,\infty}^{I_{S}(1)}$. Un tel $f$ appartient a fortiori \`a $\pi(r,0,1)^{I_{S}(1)}$ de sorte que, de nouveau gr\^ace au Th\'eor\`eme \ref{BreuilMoi}, on peut l'\'ecrire sous la forme 
\begin{equation}
f = av_{r,\infty} + bv_{r,0}
\label{inv1}
\end{equation}
avec $a$, $b \in \FFF_{p}$ non simultan\'ement nuls.\\
Par ailleurs, $f$ est un \'el\'ement de $\pi_{r,\infty} = \langle G_{S}\cdot v_{r,\infty} \rangle$ donc il peut aussi s'\'ecrire sous la forme d'une somme finie 
\begin{equation}
f = \displaystyle \sum_{i \in I} \lambda_{i}\*s_{i}\*v_{r,\infty}
\label{inv2}
\end{equation}
avec $\lambda_{i} \in \FFF_{p}$ et $s_{i} \in G_{S}$ pour tout $i \in I$.\\
En comparant les identit\'es (\ref{inv1}) et (\ref{inv2}), on obtient alors que
$$ bv_{r,0} = -av_{r,\infty} + \displaystyle \sum_{i \in I} \lambda_{i}\*s_{i}\*v_{r,\infty} \in \pi_{r,0} \cap \pi_{r, \infty} $$
ce qui n'est possible que si $b = 0$ (Lemme \ref{rescle}), i.e. si $f = av_{r,\infty}$ est colin\'eaire \`a $v_{r,\infty}$.
\end{proof}

\subsubsection{Le r\'esultat d'irr\'eductibilit\'e}
\label{superired}
La Proposition \ref{faitcrucial} nous permet maintenant d'obtenir facilement l'irr\'eductibilit\'e de nos repr\'esentations supersinguli\`eres.
\begin{prop}
Les repr\'esentations $\pi_{r,0}$ et $\pi_{r,\infty}$ sont des repr\'esentations lisses irr\'eductibles non nulles de $G_{S}$.
\label{ired}
\end{prop}
\begin{proof}
Soit $z \in \{0 , \infty\}$. La non nullit\'e de la repr\'esentation $\pi_{r,z}$ est \'evidente puisqu'elle contient le vecteur non nul $v_{r,z}$  (cf Th\'eor\`eme \ref{BreuilMoi} par exemple). De m\^eme, la lissit\'e de l'action de $G_{S}$ provient directement de celle de l'action de $G$ sur $\pi(r,0,1)$.\\
Soit maintenant $\pi$ une sous-repr\'esentation non nulle de $\pi_{r,z}$. La Proposition \ref{faitcrucial} appliqu\'ee \`a $P = I_{S}(1)$ assure que $\pi$ contient un vecteur $v \not= 0$ qui est $I_{S}(1)$-invariant. Ce vecteur appartient a fortiori \`a $\pi_{r,z}^{I_{S}(1)}$ qui est de dimension 1 sur $\FFF_{p}$ par la Proposition \ref{dim1}. On en d\'eduit donc que $v$ est colin\'eaire \`a $v_{r,z}$ avec coefficient de proportionnalit\'e non nul, ce qui implique que le vecteur $v_{r,z}$ est contenu dans $\pi$, et donc que $\pi = \pi_{r,z}$ puisque $\pi_{r,z}$ est par d\'efinition engendr\'ee par $v_{r,z}$.
\end{proof}

\subsection{D\'ecomposition en somme directe de $\pi(r,0,1) \vert_{G_{S}}$}
\label{decompo}
\begin{thm}
L'inclusion naturelle de $\pi_{r,\infty}$ et $\pi_{r,0}$ dans $\pi(r,0,1)$ induit un isomorphisme de $G_{S}$-repr\'esentations :
$$\pi(r,0,1) = \pi_{r, \infty} \oplus \pi_{r, 0} \ .$$
\label{decomp}
\end{thm}
\begin{proof}
D'apr\`es le Lemme \ref{rescle}, la somme de droite est bien une somme directe. Par suite, on dispose d'une inclusion canonique de $\pi_{r, \infty} \oplus \pi_{r, 0}$ dans $\pi(r,0,1)$ (provenant des inclusions naturelles de $\pi_{r,\infty}$ et $\pi_{r,0}$ dans $\pi(r,0,1)$) qui est clairement injective, $\FFF_{p}$-lin\'eaire et $G_{S}$-\'equivariante. Le fait que tout \'el\'ement de $\pi(r,0,1)$ peut s'\'ecrire comme somme d'un \'el\'ement de $\pi_{r,0}$ et d'un \'el\'ement de $\pi_{r,\infty}$ s'obtient par un calcul facile d\'etaill\'e dans l'Annexe \ref{calcdecompo}. 
\end{proof}

\begin{cor}
Les seules sous-$G_{S}$-repr\'esentations irr\'eductibles non nulles contenues dans $\pi(r,0,1)$ sont $\pi_{r, \infty}$ et $\pi_{r,0}$.
\label{slt}
\end{cor}
\begin{proof}
Soit $\pi$ une telle sous-repr\'esentation. D'apr\`es le Th\'eor\`eme \ref{Grep}, on peut la d\'ecomposer sous la forme $\pi = (\pi \cap \pi_{r,\infty}) \oplus (\pi \cap \pi_{r, 0}$). Chaque terme de cette somme directe est une sous-$G_{S}$-repr\'esentation de $\pi$, donc est nulle ou \'egale \`a $\pi$ par irr\'eductibilit\'e de $\pi$. Mais $\pi$ \'etant non nulle, le Lemme \ref{rescle} implique que l'un des deux membres est n\'ecessairement nul tandis que l'autre est par suite \'egal \`a $\pi$.\\
Par ailleurs, la Proposition \ref{ired} assure que $\pi_{r, \infty}$ et $\pi_{r,0}$ sont elles aussi irr\'eductibles. On en d\'eduit que $\pi \cap \pi_{r, \infty}$ (resp. $\pi \cap \pi_{r, 0}$) doit \^etre nulle ou \'egale \`a $\pi_{r, \infty}$ (resp. $\pi_{r, 0}$). La comparaison des deux expressions de $\pi \cap \pi_{r, \infty}$ (resp. $\pi \cap \pi_{r,0}$) ainsi obtenues prouve alors que l'on doit avoir $\pi = \pi_{r, \infty}$ ou $\pi = \pi_{r,0}$, ce qui termine la preuve.
\end{proof}

\subsection{Isomorphismes entre les repr\'esentations $\pi_{r,\lambda}$}
\label{isomgs}
On se propose dans cette section de d\'eterminer tous les entrelacements pouvant exister entre les repr\'esentations $\pi_{r,z}$ ($ 0 \leq r \leq p-1$, $z \in \{0, \infty\}$).\\ 
Rappelons tout d'abord que Breuil a d\'emontr\'e les r\'esultats suivants pour les repr\'esentations supersinguli\`eres de $G$ \cite[Corollaires 4.1.3 et 4.1.5]{Br1}.
\begin{thm} Soit $r \in \{0, \ldots p-1\}$.
\begin{enumerate}
\item Il existe un unique isomorphisme de $G$-repr\'esentations
$$ \pi(r,0,1) \simeq \pi(p-1-r, 0, \omega^{r})$$
envoyant $v_{r,\infty}$ sur $v_{p-1-r, 0}$.
\item Les seules $G$-repr\'esentations isomorphes \`a $\pi(r,0,1)$ sont les suivantes : 
$$\pi(r,0,1) \ ; \ \pi(r,0, \mu_{-1}) \ ; \ \pi(p-1-r, 0, \omega^{r}) \ ; \ \pi(p-1-r, 0, \mu_{-1}\omega^{r}) \ .$$
\end{enumerate}
\label{ssgiso}
\end{thm} 
On note $\omega$ le caract\`ere cyclotomique modulo $p$ vu comme caract\`ere de $\QQ_{p}^{\times}$ via l'isomorphisme de r\'eciprocit\'e (faisant correspondre les uniformisantes aux Frobenius g\'eom\'etriques) et on rappelle que $\mu_{-1}$ d\'esigne le caract\`ere non ramifi\'e de $\QQ_{p}^{\times}$ envoyant $p$ sur $-1$.\\

Remarquons ici que, par construction de nos repr\'esentations supersinguli\`eres, un isomorphisme de $G_{S}$-repr\'esentations $\pi_{r,z} \simeq \pi_{s,y}$ est uniquement d\'etermin\'e par l'image de $v_{r,z}$ qui doit \^etre un \'el\'ement de $\pi_{s,y}^{I_{S}(1)}$. La Proposition \ref{dim1} implique alors en particulier que l'espace d'entrelacements entre $\pi_{r,z}$ et $\pi_{s,y}$ est de dimension (sur $\FFF_{p}$) major\'ee par 1.\\

D\'eterminons les entrelacements existants entre deux repr\'esentations supersinguli\`eres $\pi_{r,z}$ et $\pi_{s,y}$, o\`u $(r,z)$ et $(s,y)$ sont deux paires arbitraires de param\`etres. On commence tout d'abord par remarquer que le premier point du Th\'eor\`eme \ref{ssgiso} implique l'existence d'un isomorphisme de $G_{S}$-repr\'esentations
\begin{equation} 
 \pi_{r,\infty} \simeq \pi_{p-1-r, 0}
\label{isomp1}
\end{equation}
pour tout $r \in \{0, \ldots, p-1\}$.  En faisant agir des sous-groupes bien choisis de $G_{S}$ sur de bons espaces de vecteurs invariants, nous allons montrer qu'il n'existe pas d'autre entrelacement possible. De nouveau, on fixe $r \in \{0, \ldots , p-1\}$.

\subsubsection{Action du sous-groupe d'Iwahori de $G_{S}$}
Les vecteurs $v_{r,0}$ et $v_{r,\infty}$ sont fixes sous l'action de $I_{S}(1)$, de sorte que le groupe $I_{S}$ agit sur eux via le quotient $I_{S} / I_{S}(1)$ qui est isomorphe au tore diagonal de $SL_{2}(\mathbb{F}_{p})$, lui-m\^eme isomorphe au groupe cyclique $\mathbb{F}_{p}^{\times}$. Par suite, l'action de $I_{S}$ sur $v_{r,0}$ et $v_{r,\infty}$ sera donn\'ee par des caract\`eres lisses $\mathbb{F}_{p}^{\times} \to \FFF_{p}^{\times}$.\\
Notons $\omega_{1}$ le caract\`ere fondamental de Serre de niveau 1 (\cite{Vig06}, section 2.3). Pour tout $\lambda \in \mathbb{F}_{p}^{\times}$, on a alors :
$$ \left( \begin{array}{cc} \lambda & 0 \\ 0 & \lambda^{-1} \end{array} \right) [I_{2} , x^{r}] = [I_{2}, \sigma_{r}( \left( \begin{array}{cc} \lambda & 0 \\ 0 & \lambda^{-1} \end{array} \right))(x^{r})] = \lambda^{r} [I_{2} , x^{r}] $$
de sorte que $I_{S}$ agit sur $v_{r, \infty}$ par le caract\`ere $\omega_{1}^{r}$, tandis que 
$$ \left( \begin{array}{cc} \lambda & 0 \\ 0 & \lambda^{-1} \end{array} \right) [\beta , x^{r}] = [\beta , \sigma_{r}( \left( \begin{array}{cc} \lambda^{-1} & 0 \\ 0 & \lambda \end{array} \right))(x^{r})] = \lambda^{-r} [\beta , x^{r}] $$
de sorte que $I_{S}$ agit sur $v_{r,0}$ par $\omega_{1}^{-r}$.\\
Sachant que les $I_{S}(1)$-invariants de $\pi_{r,\infty}$ et de $\pi_{r,0}$ sont de dimension 1 sur $\FFF_{p}$ (Proposition \ref{dim1}), on a donc prouv\'e le r\'esultat suivant.
\begin{prop}
Le sous-groupe d'Iwahori $I_{S}$ de $G_{S}$ agit sur $\pi_{r,\infty}^{I_{S}(1)}$ (resp.  $\pi_{r,0}^{I_{S}(1)}$) par le caract\`ere $\omega_{1}^{r}$ (resp. $\omega_{1}^{-r}$), o\`u $\omega_{1}$ d\'esigne le caract\`ere fondamental de Serre de niveau 1.
\label{actioniwahori}
\end{prop} 
Ceci nous permet d'obtenir la non-existence d'un grand nombre d'isomorphismes.
\begin{prop}
Soient $r$, $s \in \{0, \ldots, p-1\}$ avec $r \not\in \{0,  \displaystyle \frac{p-1}{2}, p-1\}$.
\begin{enumerate}
\item Si $s \not= r$, alors $\pi_{r, \infty}$ et $\pi_{s, \infty}$ (respectivement : $\pi_{r,0}$ et $\pi_{s,0}$) ne sont pas isomorphes.
\item Si $s \not= p-1-r$, alors $\pi_{r, \infty}$ et $\pi_{s, 0}$ ne sont pas isomorphes.
\end{enumerate}
\label{nonisomgener}
\end{prop}
\begin{proof}
Nous allons prouver ces deux assertions par contraposition en remarquant que si deux des repr\'esentations de l'\'enonc\'e sont isomorphes, alors $I_{S}$ doit agir par le m\^eme caract\`ere sur leurs espaces de vecteurs $I_{S}(1)$-invariants.\\
Supposons tout d'abord que $\omega_{1}^{r} = \omega_{1}^{s}$, ce qui est une condition n\'ecessaire pour que $\pi_{r, \infty}$ et $\pi_{s, \infty}$ soient des repr\'esentations isomorphes de $G_{S}$. Cela signifie que $\omega_{1}^{r-s}$ est le caract\`ere trivial, donc que $p-1$ divise $r-s$ (car $\omega_{1}$ est d'ordre $p-1$). Comme $r$ et $s$ sont tous deux compris entre $0$ et $p-1$, les seules valeurs de $r-s$ divisibles par $p-1$ sont $-(p-1)$, $0$ et $(p-1)$. Les premi\`ere et derni\`ere valeurs sont exclues car on a suppos\'e $r \not=0$ et $r \not= p-1$. On en d\'eduit donc que $r = s$, d'o\`u la premi\`ere assertion pour $\pi_{r,\infty}$ et $\pi_{s, \infty}$. Le r\'esultat pour $\pi_{r,0}$ et $\pi_{s,0}$ s'obtient de m\^eme en travaillant avec l'exposant $-r+s$.\\
Pour la seconde assertion, nous devons cette fois avoir $\omega_{1}^{r} = \omega_{1}^{-s}$, ce qui implique que $p-1$ divise $r+s$. De nouveau, comme $r$ et $s$ sont dans $\{0, \ldots ,p-1\}$, les seules valeurs de $r+s$ divisibles par $p-1$ sont $0$, $p-1$ et $2(p-1)$. La premi\`ere valeur est exclue car elle n'est atteinte que pour $r = 0 = s$ (et l'on a suppos\'e $r \not= 0$); la derni\`ere valeur est exclue car elle n'est atteinte que pour $r = p-1 = s$ (et l'on a suppos\'e $r \not= p-1$). Il faut donc que $r+s = p-1$, ce qui revient \`a dire que $s = p-1-r$.  
\end{proof}

\subsubsection{Action du sous-groupe compact maximal $K_{S}$}
\label{compacmax}
Il nous reste maintenant \`a traiter le cas o\`u $r \in \{0, \frac{p-1}{2}, p-1\}$. Lorsque $r = \frac{p-1}{2}$, l'isomorphisme (\ref{isomp1}) donne 
$$ \pi_{\frac{p-1}{2}, \infty} \simeq \pi_{\frac{p-1}{2}, 0}$$
donc les deux sous-$G_{S}$-repr\'esentations de $\pi(r,0,1)$ obtenues par le Corollaire \ref{slt} sont isomorphes.\\
Remarquons en outre que le cas $r = p-1$ se d\'eduit du cas $r = 0$ gr\^ace \`a l'isomorphisme (\ref{isomp1}), de sorte qu'il ne nous reste qu'\`a \'etudier le cas o\`u $r = 0$. Nous allons montrer que les espaces de vecteurs $K_{S}$-invariants $\pi_{0, \infty}^{K_{S}}$ et $\pi_{0,0}^{K_{S}}$ sont alors de dimension diff\'erente sur $\FFF_{p}$ : ceci assurera qu'ils ne peuvent pas \^etre isomorphes, et donc que les repr\'esentations $\pi_{0,\infty}$ et $\pi_{0,0}$ sont n\'ecessairement non-isomorphes.
\begin{prop}
\begin{enumerate}
\item $\pi_{0,0}^{K_{S}} = \{0\} \ .$
\item $\pi_{0,\infty}^{K_{S}} = \pi_{0,\infty}^{I_{S}(1)} = \FFF_{p} v_{0,\infty} \ .$
\end{enumerate}
\end{prop}
\begin{proof}
Puisque $I_{S}(1)$ est un sous-groupe de $K_{S}$, on sait que pour toute repr\'esentation $\pi$ de $G_{S}$, $\pi^{K_{S}}$ est un sous-espace vectoriel de $\pi^{I_{S}(1)}$. D'apr\`es la Proposition \ref{dim1}, il nous suffit donc de tester si $v_{0, \infty}$ et $v_{0,0}$ sont fixes sous l'action de $K_{S}$. D'apr\`es la propri\'et\'e $ii)$ rappel\'ee \`a la fin de la Section \ref{hecke}, il est clair que $v_{0, \infty}$ est fixe sous l'action de $K_{S}$, ce qui nous donne la seconde assertion.\\
D'autre part, un calcul imm\'ediat et une comparaison avec \cite[Formule (4) page 7]{Br1} montrent que $\left( \begin{array}{cc} 1 & 1 \\ -1 & 0\end{array}\right) [\beta, 1] - [\beta, 1]$ n'appartient pas \`a l'image de l'op\'erateur $T_{0}$, ce qui signifie que $v_{0,0}$ n'est pas fixe sous l'action de $\left( \begin{array}{cc} 1 & 1 \\ -1 & 0\end{array}\right)$ et prouve donc la premi\`ere assertion.
\end{proof}

\subsection{Preuve du Th\'eor\`eme \ref{ssgqp}}
Remarquons tout d'abord que les repr\'esentations $\pi_{r, \infty}$ et $\pi_{r, 0}$ sont effectivement des repr\'esentations supersinguli\`eres de $G_{S}$: 
\begin{itemize}
\item le Th\'eor\`eme \ref{decomp} et les isomorphismes (\ref{isomp1}) assurent que $\pi_{r,\infty}$ et $\pi_{r,0}$ sont de dimension infinie, donc qu'elles ne peuvent \^etre isomorphes \`a un caract\`ere;
\item la Proposition \ref{dim1} assure que les repr\'esentations $\pi_{r,\infty}$ et $\pi_{r,0}$ ne peuvent pas \^etre isomorphes \`a des s\'eries principales, pas plus qu'\`a $Ind_{B_{S}}^{G_{S}}(\mathbf{1})$, car ces repr\'esentations ont des espaces de vecteurs $I_{S}(1)$-invariants de dimension 2 sur $\FFF_{p}$ (cf. Section \ref{classifnssg});
\item enfin, les repr\'esentations $\pi_{r,\infty}$ et $\pi_{r,0}$ ne peuvent \^etre isomorphes \`a $Sp_{S}$ car ceci entrainerait (par le Th\'eor\`eme \ref{dejairred} puis par r\'eciprocit\'e de Frobenius compacte) que $\pi(r,0,1)$ serait une $G$-repr\'esentation isomorphe \`a $Sp$, ce qui contredit les r\'esultats de Breuil et Barthel-Livn\'e.
\end{itemize}

Le Th\'eor\`eme \ref{ssgqp} d\'ecoule maintenant facilement des r\'esultats prouv\'es pr\'ec\'edemment :
\begin{itemize}
\item le quatri\`eme point est exactement l'\'enonc\'e du Th\'eor\`eme \ref{decomp};
\item la Proposition \ref{ired} et le Corollaire \ref{slt} prouvent le premier point;
\item la Proposition \ref{nonisomgener} et la section \ref{compacmax} d\'emontrent les points 2 et 3.
\end{itemize}

\subsection{Vers une correspondance de Langlands modulo $p$?} 
On note $G_{\QQ_{p}} := Gal(\QQQ_{p}/\QQ_{p})$ le groupe de Galois absolu de $\QQ_{p}$, $\omega_{2}$ le caract\`ere fondamental de Serre de niveau 2, et ind($\omega_{2}^{r+1}$) l'unique repr\'esentation irr\'eductible de $G_{\QQ_{p}}$ de dimension 2 sur $\FFF_{p}$ ayant pour d\'eterminant $\omega_{2}^{r+1}$ et dont la restriction au sous-groupe d'inertie de $G_{\QQ_{p}}$ vaut $\omega_{2}^{r+1} \oplus \omega_{2}^{p(r+1)}$.\\

\noindent Dans le cas des repr\'esentations modulo $p$ de $GL_{2}(\QQ_{p})$, Breuil a d\'emontr\'e \cite[Corollaire 4.2.3]{Br1} l'existence d'une unique bijection entre les classes d'isomorphisme des repr\'esentations supersinguli\`eres de $GL_{2}(\QQ_{p})$ et les classes d'isomorphisme des repr\'esentations irr\'eductibles de dimension 2 sur $\FFF_{p}$ de $G_{\QQ_{p}}$ qui envoie $\pi(r,0,\chi)$ sur ind($\omega_{2}^{r+1}$)$\otimes \chi$ pour tout entier $r \in \{0, \ldots , p-1\}$ et tout caract\`ere $\chi : F^{\times} \to \FFF_{p}^{\times}$. Ceci lui permet de d\'efinir ce qu'il appelle une << correspondance de Langlands locale semi-simple modulo $p$ >> \cite[D\'efinition 4.2.4]{Br1} qui relie les classes d'isomorphisme de repr\'esentations semi-simples de $G_{\QQ_{p}}$ de dimension 2 sur $\FFF_{p}$ \`a certaines classes d'isomorphisme de repr\'esentations lisses semi-simples de $GL_{2}(\QQ_{p})$ sur $\FFF_{p}$.\\

\noindent Regardons maintenant ce qu'il advient lorsque l'on essaye d'\'etablir un tel lien entre repr\'esentations galoisiennes et repr\'esentations supersinguli\`eres de $SL_{2}(\QQ_{p})$. D'apr\`es le Th\'eor\`eme \ref{decomp}, on peut faire correspondre les repr\'esentations galoisiennes ind($\omega_{2}^{r+1}$) \`a des paquets de repr\'esentations supersinguli\`eres de $SL_{2}(\QQ_{p})$ : 
$$\text{ind}(\omega_{2}^{r+1}) \longleftrightarrow \{\pi_{r,\infty} ; \pi_{p-1-r, 0}\} \ . $$

Une premi\`ere diff\'erence avec le cas de $GL_{2}(\QQ_{p})$ est que les repr\'esentations supersinguli\`eres apparaissent par \emph{paquets} de taille 1 (si $r = \frac{p-1}{2}$) ou 2 (sinon). De plus, le Th\'eor\`eme \ref{ssgqp} assure que les paquets correspondant \`a $r$ et \`a $p-1-r$ sont les m\^emes, ce qui se traduit du c\^ot\'e galoisien par le fait que les repr\'esentations projectives induites par ind($\omega_{2}^{r+1}$) et ind($\omega_{2}^{p-1-(r+1)}$) sont isomorphes.\\

Une seconde diff\'erence notable avec le cas de $GL_{2}(\QQ_{p})$ est que l'on perd la possibilit\'e de distinguer la torsion par les caract\`eres du c\^ot\'e galoisien : en effet, si l'on reprend la correspondance obtenue par Breuil et qu'on la restreint \`a $SL_{2}(\QQ_{p})$, on voit que quelque soit le caract\`ere $\chi : F^{\times} \to \FFF_{p}^{\times}$ et quel que soit l'entier $r \in \{ 0, \ldots , p-1\}$, les repr\'esentations galoisiennes ind($\omega_{2}^{r+1}$) et ind($\omega_{2}^{r+1}$)$\otimes \chi$ correspondent au m\^eme paquet $\{\pi_{r,\infty} ;  \pi_{p-1-r, 0}\}$ de repr\'esentations supersinguli\`eres de $SL_{2}(\QQ_{p})$.\\

Par suite, si l'on souhaite d\'efinir une bijection analogue \`a celle obtenue pour $GL_{2}(\QQ_{p})$ et compatible \`a la restriction des repr\'esentations de $GL_{2}(\QQ_{p})$ \`a $SL_{2}(\QQ_{p})$, il faut consid\'erer la fl\`eche\footnote{On note $proj$ l'application naturelle de $GL_{2}(...)$ vers $PGL_{2}(...)$.}
$$proj \circ \text{ind}(\omega_{2}^{r+1}) \longleftrightarrow \{\pi_{r,\infty} ; \pi_{p-1-r, 0}\} $$
qui met en bijection l'ensemble des \emph{paquets} de repr\'esentations supersinguli\`eres de $G_{S}$ exhib\'es ci-dessus d'une part, et d'autre part l'ensemble des repr\'esentations galoisiennes \emph{projectives} de dimension 2 \` a isomorphisme \emph{et \`a torsion par un caract\`ere de $F^{\times}$} pr\`es.


\newpage
\appendix
\section{\'Etude des s\'eries principales << \`a la Barthel-Livn\'e >>}
\label{spBL}
L'objectif de cette annexe est de donner une d\'emonstration du Th\'eor\`eme \ref{sppal} en suivant les id\'ees de la preuve donn\'ee par Barthel-Livn\'e dans le cadre de l'\'etude des repr\'esentations en s\'eries principales de $G$ et en effectuant les modifications techniques n\'ecessaires pour les adapter au cas des repr\'esentations de $G_{S}$.\\

On commence tout d'abord par donner une autre caract\'erisation des \'el\'ements de $Ind_{B_{S}}^{G_{S}}(\eta)$ en partant de la d\'ecomposition de Bruhat :
$$ G_{S} = B_{S} \sqcup B_{S}sU \ \text{ (union disjointe)} \ .$$
Ceci implique que tout \'el\'ement $U$-invariant $f \in Ind_{B_{S}}^{G_{S}}(\eta)$ est enti\`erement d\'etermin\'e par ses valeurs en $I_{2}$ et en $s$.\\ 
On d\'efinit alors une fonctionnelle $j$ sur $Ind_{B_{S}}^{G_{S}}(\eta)$ par la formule suivante :
$$\forall \ x \in F, \  j(f)(x) := f(s u(x))$$
o\`u l'on a pos\'e $u(x) := \left( \begin{array}{cc} 1 & x \\ 0 & 1\end{array}\right) \in U$. En remarquant que l'on a 
\begin{equation}
s u(x) = \left( \begin{array}{cc} x^{-1} & -1 \\ 0 & x \end{array}\right) \left(\begin{array}{cc} 1 & 0 \\ x^{-1}& 1 \end{array}\right) \ , 
\label{dcpcalc}
\end{equation}
on obtient, puisque $f$ est localement constante, que $j(f)$ est localement constante et que, pour $x$ assez grand (i.e. $v_{F}(x)$ tr\`es n\'egative), on a \cite[page 14]{BL2}
$$ j(f)(x) = c_{f} \eta(x^{-1})$$
o\`u $c_{f} := f(I_{2})$. L'application $j$ est donc \`a valeurs dans l'ensemble $\mathcal{J}(\eta)$ des fonctions $\phi : F \to \FFF_{p}$ localement constantes pour lesquelles il existe une constante $c_{\phi}$ telle que $\phi(x) = c_{\phi} \eta(x^{-1})$ pour tout $x$ assez grand, et induit en fait clairement un isomorphisme de $\FFF_{p}$-espaces vectoriels entre $Ind_{B_{S}}^{G_{S}}(\eta)$ et $\mathcal{J}(\eta)$. Ceci permet de munir ce dernier d'une structure de $G_{S}$-module par transport de structure \`a partir de $Ind_{B_{S}}^{G_{S}}(\eta)$. On a de plus les propri\'et\'es suivantes.
\begin{lm}
Soit $\phi$ un \'el\'ement de $\mathcal{J}(\eta)$. Soient $x$, $y \in F$. Alors :
\begin{enumerate}
\item $u(y)\phi(x) = \phi(x+y)$;
\item $\alpha_{0} \phi(x) = \eta(\varpi_{F}^{-1})\phi(\varpi_{F}^{-2}x)$;
\item $\alpha_{0}^{-1} \phi(x) = \eta(\varpi_{F}) \phi(\varpi_{F}^{2}x)$.
\end{enumerate}
\label{lemtechcle}
\end{lm}
\begin{proof}
Ce r\'esultat est l'analogue de \cite[Lemma 17]{BL2} et se prouve par calcul direct. Pour le premier point, on \'ecrit simplement que si $\phi = j(f)$, alors
$$ u(y) \phi(x) := u(y) f(su(x)) = f(su(x)u(y)) = f(su(x+y)) = \phi(x+y)$$
tandis que les deux derniers points sont cons\'equences du calcul plus g\'en\'eral suivant : $\forall \lambda \in F^{\times}$,
$$\begin{array}{rcll} \left( \begin{array}{cc} \lambda & 0 \\ 0 & \lambda^{-1 } \end{array}\right) \phi(x)  & = & f(su(x) \left( \begin{array}{cc} \lambda & 0 \\ 0 & \lambda^{-1 } \end{array}\right)) & \\ 
& = & f( \left( \begin{array}{cc} x^{-1} & -1 \\ 0 & x \end{array}\right) \left( \begin{array}{cc} 1 & 0 \\ x^{-1} & 1 \end{array}\right) \left( \begin{array}{cc} \lambda & 0 \\ 0 & \lambda^{-1 } \end{array}\right))& \text{ par (\ref{dcpcalc})}\\
& = & \eta(x^{-1}) f(\left( \begin{array}{cc} \lambda & 0 \\ x^{-1}\lambda  & \lambda^{-1 } \end{array}\right))& \\
& = & \eta(x^{-1}) f( \left( \begin{array}{cc} 1 & x \\ 0 & 1 \end{array}\right) \left( \begin{array}{cc} 0 & -\lambda^{-1}x \\ \lambda x^{-1} & \lambda^{-1 } \end{array}\right))& \\
& = & \eta(x^{-1}) f(\left( \begin{array}{cc} 0 & -\lambda^{-1}x \\ \lambda x^{-1} & \lambda^{-1 } \end{array}\right))& \\
& = & \eta(x^{-1}) f(\left( \begin{array}{cc} \lambda^{-1}x & 0 \\ 0 & \lambda x^{-1} \end{array}\right) su(\lambda^{-2}x))& \\
& = & \eta(\lambda^{-1})\phi(\lambda^{-2}x) \ .
\end{array}$$

\end{proof}

\noindent Nous allons maintenant devoir distinguer selon que le caract\`ere $\eta$ est ramifi\'e ou non.

\subsection{Le cas non ramifi\'e}
\label{nram}
On suppose tout d'abord que le caract\`ere $\eta$ est non ramifi\'e,  i.e. qu'il est trivial sur $\mathcal{O}_{F}^{\times}$, et l'on pose $ \Lambda := \eta(\varpi_{F}^{-1})$. On note $f_{0} \in Ind_{B_{S}}^{G_{S}}(\eta)$ l'unique \'el\'ement $K_{S}$-invariant valant $1$ en $I_{2}$ et  l'on pose $\phi_{0} := j(f_{0})$. Un calcul direct fournit facilement l'\'identit\'e suivante : 
$$\phi_{0}(x) = \left\{ \begin{array}{ll} 1 & \text{ si } v_{F}(x) \geq 0 \ ; \\ \Lambda^{v_{F}(x)} & \text{ si }v_{F}(x) \leq 0 \ . \end{array} \right.$$
On dispose alors d'un analogue un peu plus pouss\'e de \cite[Proposition 18.1]{BL2} qui n\'ecessite d'introduire la notation suivante : pour tout entier $i \geq 1$, on note $\mathcal{R}_{i} \subset \mathcal{O}_{F}$ un syst\`eme de repr\'esentants de $ \mathcal{O}_{F} / \varpi_{F}^{i}\mathcal{O}_{F}$.
\begin{lm}
On a les deux identit\'es suivantes : 
\begin{enumerate}
\item $\displaystyle \sum_{x \in \mathcal{R}_{1}} n(x/\varpi_{F}) \phi_{0} = (1 - \Lambda^{-1}) \mathbf{1}_{\varpi_{F}^{-1} \mathcal{O}_{F}}$;
\item $\displaystyle \sum_{x \in \mathcal{R}_{2}} n(x/\varpi_{F}^{2}) \phi_{0} = (1 - \Lambda^{-1}) \mathbf{1}_{\varpi_{F}^{-2} \mathcal{O}_{F}}$.
\end{enumerate}
\label{lemtech}
\end{lm}

\begin{proof}
Tout comme dans la preuve de Barthel-Livn\'e, on commence par remarquer que pour tout $y \in F$, on a 
$$ \left\{ \begin{array}{l}\displaystyle \sum_{x \in \mathcal{R}_{1}} n(x/\varpi_{F}) \phi_{0}(y) = \sum_{x \in \mathcal{R}_{1}} \phi_{0}(y + x/\varpi_{F}) \ ; \\
\displaystyle \sum_{x \in \mathcal{R}_{2}} n(x/\varpi_{F}^{2}) \phi_{0}(y) = \sum_{x \in \mathcal{R}_{2}} \phi_{0}(y + x/\varpi_{F}^{2}) \ . \end{array} \right.$$ 
A partir de l\`a, il suffit de nouveau de distinguer selon la valuation $F$-adique de $y $ et de calculer. Pour le premier point, on doit traiter 3 cas distincts :  
\begin{itemize}
\item si $v_{F}(y) < -1$, alors on a $v_{F}(y + x/\varpi_{F}) = v_{F}(y) \leq 0$ pour tout $x \in \mathcal{R}_{1}$, de sorte que l'on a 
$$\begin{array}{rcll}  \displaystyle \sum_{x \in \mathcal{R}_{1}} \phi_{0}(y + x/\varpi_{F}) & = & \displaystyle \sum_{x \in \mathcal{R}_{1}}\Lambda^{v_{F}(y)}& \\
& = & q \Lambda^{v_{F}(y)} & \\
& = & 0 \\
& = & (1 - \Lambda^{-1}) \mathbf{1}_{\varpi_{F}^{-1} \mathcal{O}_{F}}(y);
\end{array}$$
\item si $v_{F}(y) > -1$, alors on a $v_{F}(y) \geq 0$ donc $v_{F}(y + x/\varpi_{F}) = v_{F}(x/\varpi_{F}) = -1$ si $x \not=0$, de sorte que l'on a cette fois
$$\begin{array}{rcll}  \displaystyle \sum_{x \in \mathcal{R}_{1}} \phi_{0}(y + x/\varpi_{F}) & = & \phi_{0}(y)+ \displaystyle \sum_{\stackrel{x \in \mathcal{R}_{1}}{x \not=0}} \Lambda^{-1}& \\
& = & 1 + (q-1) \Lambda^{-1} & \\
& = & 1 - \Lambda^{-1} \\
& = & (1 - \Lambda^{-1}) \mathbf{1}_{\varpi_{F}^{-1} \mathcal{O}_{F}}(y);
\end{array}$$
\item si $v_{F}(y) = -1$, alors il existe un unique $x_{0} \in \mathcal{R}_{1}$ tel que $x_{0} + \varpi_{F}y \equiv 0 \ [\varpi_{F}]$. On a alors :
$$\begin{array}{rcll}  \displaystyle \sum_{x \in \mathcal{R}_{1}} \phi_{0}(y + x/\varpi_{F}) & = & \phi_{0}(y + x_{0}/\varpi_{F})+ \displaystyle \sum_{\stackrel{ x \in \mathcal{R}_{1}}{x \not=x_{0}}} \Lambda^{-1}& \\
& = & 1 + (q-1) \Lambda^{-1} & \\
& = & 1 - \Lambda^{-1} \\
& = & (1 - \Lambda^{-1}) \mathbf{1}_{\varpi_{F}^{-1} \mathcal{O}_{F}}(y)
\end{array}$$
ce qui ach\`eve de prouver la premi\`ere identit\'e.\\
\end{itemize}

Pour le second point, on effectue un raisonnement similaire, en distinguant un cas suppl\'ementaire :
\begin{itemize}
\item si $v_{F}(y) < -2$, alors on a $v_{F}(y + x/\varpi_{F}^{2}) = v_{F}(y) \leq 0$ pour tout $x \in \mathcal{R}_{2}$, de sorte que l'on a 
$$\begin{array}{rcll}  \displaystyle \sum_{x \in \mathcal{R}_{2}} \phi_{0}(y + x/\varpi_{F}^{2}) & = & \displaystyle \sum_{x \in \mathcal{R}_{2}} \Lambda^{v_{F}(y)}& \\
& = & q^{2} \Lambda^{v_{F}(y)} & \\
& = & 0 \\
& = & (1 - \Lambda^{-1}) \mathbf{1}_{\varpi_{F}^{-2} \mathcal{O}_{F}}(y) \ ;
\end{array}$$

\item si $v_{F}(y) > -1$, il faut faire un peu plus attention. On a toujours $v_{F}(y) \geq 0$, de sorte que  
$$ v_{F}(y + x/\varpi_{F}^{2}) = \left\{ \begin{array}{ll} -2 & \text{ si } x \not\equiv 0 \ [\varpi_{F}] \ ; \\
-1 & \text{ si } x \equiv 0 \ [\varpi_{F}] \text{ mais } x \not=0  \ ; \\
v_{F}(y) & \text{ si } x = 0 \ .
\end{array} \right.$$

\noindent Par suite, on a donc cette fois 
$$\begin{array}{rcll}  \displaystyle \sum_{x \in \mathcal{R}_{2}} \phi_{0}(y + x/\varpi_{F}) & = & \phi_{0}(y)+ \displaystyle \sum_{\stackrel{x \in \mathcal{R}_{2}}{x \not=0 \text{ et } x \equiv 0 [\varpi_{F}]}} \Lambda^{-1} +  \sum_{\stackrel{x \in \mathcal{R}_{2}}{ x \not\equiv 0 [\varpi_{F}]}} \Lambda^{-2}& \\
& = & 1 + (q-1) \Lambda^{-1} + (q^{2} - q) \Lambda^{-2}& \\
& = & 1 - \Lambda^{-1} \\
& = & (1 - \Lambda^{-1}) \mathbf{1}_{\varpi_{F}^{-2} \mathcal{O}_{F}}(y) \ .
\end{array}$$

\item Si $v_{F}(y) = -2$, alors on \'ecrit $y = \varpi_{F}^{-2}y_{0}$ avec $v_{F}(y_{0}) = 0$, de sorte que l'on a 
$$\begin{array}{rcl}  \displaystyle \sum_{x \in \mathcal{R}_{2}} \phi_{0}(y + x/\varpi_{F}^{2}) & = & \displaystyle \sum_{x \in \mathcal{R}_{2}} \phi_{0}((y_{0} + x)/\varpi_{F}^{2}) \\
& = & 1 + (q-1) \Lambda^{-1} + (q^{2} - q) \Lambda^{-2} \\
& = & 1 - \Lambda^{-1} \\
& = & (1 - \Lambda^{-1}) \mathbf{1}_{\varpi_{F}^{-2} \mathcal{O}_{F}}(y) \\ 
\end{array}$$
o\`u le passage de la premi\`ere ligne \`a la seconde ligne s'effectue par le m\^eme calcul que pour $v_{F}(y) > -1$.\\

\item Si $v_{F}(y) = -1$, on \'ecrit cette fois $y = \varpi_{F}^{-1}y_{0}$ avec $v_{F}(y_{0}) = 0$, de sorte que l'on a comme pr\'ec\'edemment 
$$\begin{array}{rcll}  \displaystyle \sum_{x \in \mathcal{R}_{2}} \phi_{0}(y + x/\varpi_{F}^{2}) & = & \displaystyle \sum_{x \in \mathcal{R}_{2}} \phi_{0}((y_{0} + x\varpi_{F})/\varpi_{F}^{2})& \\
& = & 1 + (q-1) \Lambda^{-1} + (q^{2} - q) \Lambda^{-2} & \\
& = & 1 - \Lambda^{-1} \\
& = & (1 - \Lambda^{-1}) \mathbf{1}_{\varpi_{F}^{-2} \mathcal{O}_{F}}(y)
\end{array}$$
ce qui ach\`eve de prouver la seconde assertion.
\end{itemize}
\end{proof}

\noindent Ceci nous permet d'obtenir \`a l'aide du Lemme \ref{lemtechcle} un corollaire qui sera important dans la suite.
\begin{cor}
Si $\eta$ est un caract\`ere non ramifi\'e non trivial, alors $f_{0}$ engendre le $G_{S}$-module $Ind_{B_{S}}^{G_{S}}(\eta)$.
\label{genphi}
\end{cor}
\begin{proof}
 Dire que $\eta$ est non trivial \'equivaut \`a dire que l'on a $\Lambda \not= 1$ (puisque $\eta$ est suppos\'e non ramifi\'e). Dans ce cas, les Lemmes \ref{lemtechcle} et \ref{lemtech} assurent que l'\'el\'ement $\phi_{0}$ engendre le $G_{S}$-module $\mathcal{J}(\eta)$, ce qui prouve notre r\'esultat \'etant donn\'e que l'on a $f_{0} = j^{-1}(\phi_{0})$.
\end{proof}
\begin{NB}
Rappelons que, par la d\'ecomposition d'Iwasawa $G_{S} = UT_{S}K_{S}$, l'espace des $K_{S}$-invariants de $Ind_{B_{S}}^{G_{S}}(\eta)$ est de dimension 1 sur $\FFF_{p}$, et que $\{f_{0}\}$ en est de fait une base.
\end{NB}
Avant de prouver le Th\'eor\`eme \ref{sppal} dans le cas non ramifi\'e, on rappelle que l'on dispose d'une autre d\'ecomposition de $G_{S}$ \cite[Section 3.3]{BL2}:
$$ G_{S} = B_{S}I_{S} \sqcup B_{S}s I_{S}$$
qui peut \^etre r\'e\'ecrite $G_{S} = B_{S}I_{S} \sqcup B_{S} \beta_{0} I_{S}$. Ceci nous permet de prouver facilement que $(Ind_{B_{S}}^{G_{S}}(\eta))^{I_{S}}$ est de dimension 2 sur $\FFF_{p}$ et d'en donner une base explicite, \`a savoir la famille $\{f_{1}, f_{2}\}$ d\'efinie par 
$$ \left\{ \begin{array}{ccc} f_{1}(I_{2}) = 1 & ; & f_{1}(\beta_{0}) = 0 \ ; \\ f_{2}(I_{2}) = 0 & ; & f_{2}(\beta_{0}) = 1 \ . \end{array}\right.$$
La fonction $f_{0}$ \'etant $K_{S}$-invariante, elle est en particulier $I_{S}$-invariante donc peut se d\'ecomposer dans la base $\{f_{1},f_{2}\}$. Par d\'efinition, on a $f_{0}(I_{2}) = 1$ et, puisque $\beta_{0} = - \beta_{0}^{-1} = \alpha_{0}^{-1} s$, on obtient facilement que $f_{0}(\beta_{0}) = \eta(\varpi_{F}^{-1}) = \Lambda$. On en conclut donc que l'on a 
\begin{equation}
f_{0} = f_{1} + \Lambda f_{2}
\label{dcpf}
\end{equation}
Remarquons aussi que l'on a tout aussi facilement les identit\'es suivantes : 
\begin{equation}
\left\{ \begin{array}{ccc} \alpha_{0}f_{1} = \Lambda^{-1}f_{1} & ; & \beta_{0}f_{1} = f_{2} \ ; \\ \alpha_{0}f_{2} = \Lambda f_{2} & ; & \beta_{0} f_{2} = f_{1} \ . \\  \end{array}\right.
\label{actionI}
\end{equation}

\vspace{\baselineskip}

\noindent \textbf{Preuve du Th\'eor\`eme \ref{sppal} dans le cas non ramifi\'e :} Soit $V$ une sous-repr\'esentation non nulle de $Ind_{B_{S}}^{G_{S}}(\eta)$. D'apr\`es la Proposition \ref{faitcrucial}, l'espace $V^{I_{S}(1)}$ est non nul. Puisque $\eta$ est non ramifi\'e, le sous-groupe d'Iwahori agit trivialement sur $(Ind_{B_{S}}^{G_{S}}(\eta))^{I_{S}(1)}$, donc en particulier sur $V^{I_{S}(1)}$, de sorte que l'on a aussi $V^{I_{S}} \not= \{0\}$.\\
Soit alors $f$ un \'el\'ement non nul de $V^{I_{S}}$ : il est contenu dans $(Ind_{B_{S}}^{G_{S}}(\eta))^{I_{S}}$ donc il s'\'ecrit de mani\`ere unique sous la forme $f = a_{1}f_{1} + a_{2}f_{2}$ avec $a_{1}$, $a_{2} \in \FFF_{p}$ non simultan\'ement nuls. Si $a_{1} = 0$ (respectivement : $a_{2} = 0$), alors $f$ est colin\'eaire \`a $f_{1}$ (resp. $f_{2}$) avec coefficient de proportionnalit\'e non nul, de sorte que $f_{1}$ (resp. $f_{2}$) est contenu dans $V$. D'apr\`es (\ref{actionI}), on obtient qu'alors $f_{2}$ (resp. $f_{1}$) est contenu dans $V$, donc que $f_{0} = f_{1} + \Lambda f_{2}$ est contenu dans $V$, et donc que $V = Ind_{B_{S}}^{G_{S}}(\eta)$ d'apr\`es le Corollaire \ref{genphi}.\\
Supposons maintenant que $a_{1}a_{2} \not=0$ et m\^eme, quitte \`a renormaliser $f$, que $a_{1} = 1$. On a alors 
\begin{equation}
 \left\{ \begin{array}{l} \alpha_{0}f = \Lambda^{-1} f_{1} + a_{2} \Lambda f_{2} \ , \\ 
 \beta_{0}f = f_{2} + a_{2}f_{1} \ ,
 \end{array} \right.
 \label{actionsurf}
 \end{equation}
de sorte que $h := \Lambda\alpha_{0}f - f = (\Lambda^{2} - 1)a_{2}f_{2}$ est encore un \'el\'ement de $V$. Par suite, si l'on a $\Lambda^{2} \not=1$, on obtient que $f_{2}$ est contenu dans $V$, ce qui nous permet de conclure comme pr\'ec\'edemment.\\
Supposons que $\Lambda^{2} = 1$. Comme $\eta$ est suppos\'e non trivial et non ramifi\'e, on a n\'ecessairement $\Lambda \not=1$, donc on peut supposer que $p \not= 2$ et $\Lambda = -1$, ce qui revient \`a dire que $\eta$ est l'unique caract\`ere quadratique non ramifi\'e de $F^{\times}$. On a alors $f_{0} = f_{1} - f_{2}$, et le syst\`eme (\ref{actionsurf}) devient alors 
$$ \left\{ \begin{array}{l} \alpha_{0}f = - f_{1} - a_{2} f_{2} = -f \ ; \\ 
 \beta_{0}f = f_{2} + a_{2}f_{1} \ .
 \end{array} \right. $$

\noindent Trois cas se pr\'esentent :
\begin{itemize}
\item ou bien $a_{2}^{2} \not= 1$, auquel cas $\beta_{0} f - a_{2}f = (1- a_{2}^{2})f_{2}$ est contenu dans $V$ avec $1 - a_{2}^{2} \not= 1$, donc $f_{2} \in V$ et on peut conclure comme ci-dessus.
\item Ou bien $a_{2} = -1$, auquel cas $f = f_{1} - f_{2} = f_{0}$ est contenu dans $V$, ce qui implique directement que $V = Ind_{B_{S}}^{G_{S}}(\eta)$ d'apr\`es le Corollaire \ref{genphi}.
\item Ou bien $a_{2} = 1$, auquel cas on a  $f = f_{1} + f_{2} = \beta_{0}f$. Un calcul direct \`a partir de (\ref{actionI}) et de l'identit\'e $s = \alpha_{0}\beta_{0}$ montre alors que $s \cdot f_{1} = f_{2}$  et $s \cdot f_{2} = -f_{1}$, donc que $s \cdot f = f_{2} - f_{1} = - f_{0}$. Ceci assure que $f_{0}$ est contenu dans la repr\'esentation de $G_{S}$ engendr\'ee par $f$, elle-m\^eme contenue dans $V$, donc que $f_{0}$ est contenu dans $V$, ce qui permet de terminer la d\'emonstration. \\
\end{itemize}

\subsection{Le cas ramifi\'e}
\label{ram}
Supposons maintenant que $\eta$ soit un caract\`ere ramifi\'e. Nous perdons l'existence de la fonction $f_{0}$, mais nous disposons encore de fonctions analogues aux fonctions $f_{1}$ et $f_{2}$ d\'efinies dans la section pr\'ec\'edente. Plus pr\'ecis\'ement, la d\'ecomposition $G_{S} = B_{S}I_{S}(1) \sqcup B_{S}\beta_{0}I_{S}(1)$ assure que l'espace $(Ind_{B_{S}}^{G_{S}}(\eta))^{I_{S}(1)}$ des vecteurs $I_{S}(1)$-invariants est encore de dimension 2 sur $\FFF_{p}$ (cf. aussi \cite[Lemma 28]{BL1}) et qu'une base en est donn\'ee par la famille de fonctions $\{\ell_{1}  , \ell_{2}\}$ caract\'eris\'ees par  :
\begin{equation}
\left\{ \begin{array}{cccl} \ell_{1}(I_{2}) = 1 & ; & \ell_{1}(\beta_{0}) = 0 & ; \\ \ell_{2}(I_{2}) = 0 & ; & \ell_{2}(\beta_{0}) = 1 & .  \end{array}\right.
\label{ellcond}
\end{equation}
Elles satisfont les m\^emes relations que leurs analogues du cas non ramifi\'e : 
\begin{equation}
\left\{ \begin{array}{cccl} \alpha_{0}\ell_{1} = \Lambda^{-1}\ell_{1} & ; & \beta_{0}\ell_{1} = \ell_{2} & ; \\ \alpha_{0}\ell_{2} = \Lambda \ell_{2} & ; & \beta_{0} \ell_{2} = \ell_{1} & . \\  \end{array}\right.
\label{actionell1}
\end{equation}
Le r\'esultat suivant, qui fait de nouveau appel \`a l'isomorphisme $j$, nous permet alors de reprendre les arguments du cas non ramifi\'e pour prouver le Th\'eor\`eme \ref{sppal} dans le cas ramifi\'e.
\begin{lm} Soit $x \in F$. Alors : 
\begin{enumerate}
\item $j(\ell_{1})(x) =  \left\{ \begin{array}{ll}  \eta(x^{-1})  & \text{ si }v_{F}(x) < 0 \ ; \\ 0  & \text{ si } v_{F}(x) \geq 0 \ . \end{array}\right.$ 
\item $j(\ell_{2})(x) = \mathbf{1}_{\mathcal{O}_{F}}(x)$ .
\end{enumerate}
\label{techramphi}
\end{lm}

\begin{proof}
Pour toute fonction $f \in Ind_{B_{S}}^{G_{S}}(\eta)$, on a par d\'efinition 
$$ j(f)(x) = f(su(x)) = f(\left( \begin{array}{cc} 0 & -1 \\ 1 & x \end{array}\right)) = \eta(x^{-1})f(\left( \begin{array}{cc} 1 & 0 \\ x^{-1} & 1\end{array}\right)) \ .$$
Ceci donne la formule voulue pour $j(\ell_{1})(x)$ lorsque $v_{F}(x) <0$ puisque dans ce cas, $\left( \begin{array}{cc} 1 & 0 \\ x^{-1} & 1\end{array}\right)$ est un \'el\'ement de $I_{S}(1)$ qui fixe donc $\ell_{1}$. Dans les autres cas, on utilise tout simplement le fait que 
$$su(x) = (\alpha_{0}\beta_{0})u(x) $$
avec $\alpha_{0} \in B_{S}$ et $u(x) \in I_{S}(1)$ si $v_{F}(x) \geq 0$, ce qui prouve les r\'esultats restants gr\^ace \`a (\ref{ellcond}).
\end{proof}

\begin{lm} On a l'identit\'e suivante dans $\mathcal{J}(\eta)$ :
$$ \displaystyle \sum_{x \in \mathcal{R}_{1}} \left( \begin{array}{cc} \varpi_{F} & x \\ 0 & \varpi_{F}^{-1}\end{array}\right) \mathbf{1}_{\mathcal{O}_{F}} = \mathbf{1}_{\varpi_{F}\mathcal{O}_{F}} \ .$$
\label{cleramphi}
\end{lm}
\begin{proof}
Il suffit de fixer $y \in F$ et de faire un calcul direct comme dans le cas non ramifi\'e en faisant attention \`a la valeur de $v_{F}(y)$ : 
$$\begin{array}{rcl} \displaystyle \sum_{x \in \mathcal{R}_{1}} \left( \begin{array}{cc} \varpi_{F} & x \\ 0 & \varpi_{F}^{-1}\end{array}\right) \mathbf{1}_{\mathcal{O}_{F}}(y) & = & \displaystyle \sum_{x \in \mathcal{R}_{1}} \left( \begin{array}{cc} \varpi_{F} & x \\ 0 & \varpi_{F}^{-1}\end{array}\right) j(\ell_{2})(y)  \\
& := & \displaystyle \sum_{x \in \mathcal{R}_{1}} \ell_{2}(su(y)\left( \begin{array}{cc} \varpi_{F} & x \\ 0 & \varpi_{F}^{-1}\end{array}\right) ) \\
& = & \displaystyle \sum_{x \in \mathcal{R}_{1}}   \ell_{2}(\beta_{0}\left( \begin{array}{cc} 1 & \varpi_{F}^{-1}x + \varpi_{F}^{-2}y \\ 0 & 1\end{array}\right) )\\
& = &  \displaystyle \sum_{x \in \mathcal{R}_{1}} \mathbf{1}_{\mathcal{O}_{F}}(\varpi_{F}^{-1}x + \varpi_{F}^{-2}y)\\
& = &  \displaystyle \sum_{x \in \mathcal{R}_{1}} \mathbf{1}_{\varpi_{F}^{2}\mathcal{O}_{F}}(\varpi_{F}x + y) \ .\\
\end{array}$$
Notons que la quatri\`eme \'egalit\'e provient du Lemme \ref{lemtechcle}. On a donc de nouveau trois cas \`a distinguer :
\begin{itemize}
\item Si $v_{F}(y) \leq 0$, alors $v_{F}(\varpi_{F} x + y) = v_{F}(y) \leq 0$ pour tout $x \in \mathcal{R}_{1}$ donc la somme est nulle, tout comme l'est $\mathbf{1}_{\varpi_{F}\mathcal{O}_{F}}(y)$.
\item Si $v_{F}(y) \geq 2$, alors $v_{F}(\varpi_{F} x + y) = 1$ si $x$ est non nul et $v_{F}(\varpi_{F} x + y) = v_{F}(y)$ si $x$ est nul, donc la somme vaut $1 = \mathbf{1}_{\varpi_{F}\mathcal{O}_{F}}(y)$.
\item Si $v_{F}(y) = 1$, alors il existe un unique $x_{0} \in \mathcal{R}_{1}$ tel que $y + \varpi_{F}x_{0} \equiv 0 \ [\varpi_{F}^{2}]$. On a alors $v_{F}(\varpi_{F} x + y)= 2$ si $x = x_{0}$ et $v_{F}(\varpi_{F} x + y) = 1$ si $x \not= x_{0}$, de sorte que la somme est de nouveau \'egale \`a $1 = \mathbf{1}_{\varpi_{F}\mathcal{O}_{F}}(y)$.
\end{itemize}
\end{proof}

\begin{cor}
La famille $\{\ell_{1}, \ell_{2}\}$ engendre la $G_{S}$-repr\'esentation $Ind_{B_{S}}^{G_{S}}(\eta)$.
\label{generam}
\end{cor}
\begin{proof} Remarquons que l'on peut r\'e\'ecrire $j(\ell_{1})$ sous la forme 
$$ j(\ell_{1}) = \eta^{-1} (\mathbf{1} - \mathbf{1}_{\mathcal{O}_{F}}) \ .$$
Par suite, les Lemmes \ref{lemtechcle}, \ref{techramphi} et \ref{cleramphi} assurent que $\{j(\ell_{1}), j(\ell_{2})\}$ engendre le $G_{S}$-module $\mathcal{J}(\eta)$, ce qui \'equivaut \`a dire que $\{\ell_{1},\ell_{2}\}$ engendre le $G_{S}$-module $Ind_{B_{S}}^{G_{S}}(\eta)$.
\end{proof}

\noindent \textbf{Preuve du Th\'eor\`eme \ref{sppal} dans le cas ramifi\'e :}
Soit $V$ une sous-repr\'esentation non nulle de $Ind_{B_{S}}^{G_{S}}(\eta)$. D'apr\`es la Proposition \ref{faitcrucial}, l'espace $V^{I_{S}(1)}$ est non nul. Soit alors $f$ un \'el\'ement non nul de $V^{I_{S}(1)}$ : il est contenu dans $(Ind_{B_{S}}^{G_{S}}(\eta))^{I_{S}(1)}$ donc peut s'\'ecrire de mani\`ere unique sous la forme $f = a_{1}\ell_{1} + a_{2}\ell_{2}$ avec $a_{1}$, $a_{2} \in \FFF_{p}$ non simultan\'ement nuls. Si $a_{1} = 0$ (respectivement : $a_{2} = 0$), alors $f$ est colin\'eaire \`a $\ell_{1}$ (resp. $\ell_{2}$) avec coefficient de proportionnalit\'e non nul, de sorte que $\ell_{1}$ (resp. $\ell_{2}$) est contenu dans $V$. D'apr\`es (\ref{actionell1}), on obtient qu'alors $\ell_{2}$ (resp. $\ell_{1}$) est contenu dans $V$ et donc que $V = Ind_{B_{S}}^{G_{S}}(\eta)$ d'apr\`es le Corollaire \ref{generam}.\\
Supposons maintenant que $a_{1}a_{2} \not=0$ et m\^eme, quitte \`a renormaliser $f$, que $a_{1} = 1$. On a alors, pour tout $t \in \mathcal{O}_{F}^{\times}$,
\begin{equation}
\left( \begin{array}{cc} t & 0 \\ 0 & t^{-1} \end{array}\right) f =  \eta(t)a_{1}\ell_{1} + \eta(t^{-1}) a_{2}\ell_{2} \ ,
 \label{actionsurf}
 \end{equation}
ce qui implique que 
 $$\left( \begin{array}{cc} t & 0 \\ 0 & t^{-1} \end{array}\right) f - \eta(t) f = (\eta(t^{-1})-\eta(t))a_{2}\ell_{2} \ . $$
Par suite, s'il existe $t \in \mathcal{O}_{F}$ tel que $\eta(t^{2}) \not=1$, on obtient (pour cette valeur de $t$) que $\ell_{2}$ est contenu dans $V$, donc de nouveau que $V = Ind_{B_{S}}^{G_{S}}(\eta)$ gr\^ace au Corollaire \ref{generam}.\\
Si $\eta^{2}(t) = 1$ pour tout $t \in \mathcal{O}_{F}^{\times}$, on peut alors reprendre les arguments du cas non ramifi\'e pour conclure puisque l'on a toujours
\begin{equation}
 \left\{ \begin{array}{l} s \cdot (\ell_{1} + a_{2} \ell_{2}) = \ell_{2} - a_{2}\ell_{1} \ ; \\ 
 \beta_{0}f = \ell_{2} + a_{2}\ell_{1} \ .
 \end{array} \right.
 \label{actionsurfram}
 \end{equation}
Si $p\not= 2$, on obtient directement que $2 \ell_{2} = s \cdot f + \beta_{0} f \in V$, donc que $\ell_{2} \in V$, ce qui permet d'obtenir que $V = Ind_{B_{S}}^{G_{S}}(\eta)$ par le Corollaire \ref{generam} et par (\ref{actionell1}).\\
Si $p = 2$, l'hypoth\`ese $(\eta\vert_{\mathcal{O}_{F}^{\times}})^{2} = 1$ implique alors que $\eta\vert_{\mathcal{O}_{F}^{\times}}$ est trivial, i.e. que $\eta$ est non ramifi\'e, ce qui contredit notre hypoth\`ese de d\'epart et permet de terminer la d\'emonstration.\\
\newpage

\section{D\'etails de la preuve du Th\'eor\`eme \ref{sppal}}
\label{sppaljac}
\noindent Le but de cette annexe est de d\'evelopper la seconde partie de la Section \ref{irredspal} en adaptant les id\'ees d\'evelopp\'ees dans le cas complexe (cf. par exemple \cite[Section 9]{BH}) au cadre des repr\'esentations modulo $p$. Avant de commencer, on introduit la notation pratique suivante : pour tout $x \in F$, on pose $u(x) := \left( \begin{array}{cc} 1 & x \\ 0 & 1\end{array}\right) \in U$.

\subsection{Rappels et r\'esultats pr\'eliminaires}
On commence par rappeler la d\'efinition du foncteur de Jacquet associ\'e au sous-groupe unipotent $U$. Si $(\pi, V)$ est une repr\'esentation de $G_{S}$ sur $\FFF_{p}$, on note $V(U)$ le $\FFF_{p}$-espace vectoriel engendr\'e par l'ensemble $\{ \pi(u)v - v$; \  $v \in V$, \  $u \in U \}$. Il est clair que cet espace est stable sous l'action de $U$ (par $\pi$); par suite, l'espace quotient $V_{U} := V / V(U)$ (que l'on appelle \emph{module de Jacquet de $V$}) est naturellement muni d'une action du tore diagonal $T_{S} \simeq B_{S} / U$ via $\pi$ (que l'on notera $\pi_{U}$). L'application $[(\pi, V) \mapsto (\pi_{U}, V_{U})]$ d\'efinit alors un foncteur $J_{U}$ allant de la cat\'egorie des repr\'esentations lisses de $G$ vers la cat\'egorie des repr\'esentations lisses de $T$, que l'on appelle le \emph{foncteur de Jacquet (associ\'e \`a $U$)}.\\
Les m\^emes arguments que ceux utilis\'es dans le cas classique (voir par exemple \cite[Section 9]{BH} ou \cite[Proposition 4.4.2]{Bu}) montrent que le foncteur $J_{U}$ est exact \`a droite.

\begin{rem}
Remarquons ici qu'il est facile de voir que si $\chi$ est un caract\`ere lisse de $T_{S}$, on a un isomorphisme naturel $J_{U}(\chi) \simeq \chi$.
\label{JUcar}
\end{rem}

\subsection{Etude de la restriction \`a $B_{S}$}
On se donne maintenant un caract\`ere lisse $\eta : F^{\times} \to \FFF_{p}^{\times}$. On note encore $\eta = \eta \otimes \mathbf{1}$ le caract\`ere lisse de $B_{S}$ obtenu par inflation. L'application d'\'evaluation en $I_{2}$ d\'efinit un morphisme surjectif $B_{S}$-\'equivariant 
$$ \phi : Ind_{B_{S}}^{G_{S}}(\eta) \twoheadrightarrow \eta \ .$$
La d\'ecomposition de Bruhat $G_{S} = B_{S} \sqcup B_{S}sU$ permet d'identifier le noyau $V$ de ce morphisme au sous-espace des fonctions de $Ind_{B_{S}}^{G_{S}}(\eta)$ \`a support dans $B_{S}sU$.
\begin{prop}
$V$ est une repr\'esentation lisse irr\'eductible de $B_{S}$.
\end{prop}
\begin{proof}
On note $C^{\infty}_{c}(U)$ l'ensemble des fonctions lisses $U \to \FFF_{p}$ \`a support compact. Il est naturellement muni d'une action lisse de $B_{S}$ d\'efinie par
$$\left( \begin{array}{cc} a & b \\ 0 & a^{-1}\end{array}\right) \cdot \phi := [u(y) \mapsto \phi\left(u\left(\frac{a^{-1}y + b}{a}\right)\right)] $$
et l'application $\left[f \mapsto [u \mapsto f(su)]\right]$ induit un isomorphisme de repr\'esentations de $B_{S}$ :
\begin{equation}
\Psi :  V \simeq C^{\infty}_{c}(U)\otimes \eta^{-1} \ .
\label{isomcool}
\end{equation}
En effet, le seul point d\'elicat \`a v\'erifier est la compatibilit\'e aux actions de $B_{S}$. Pour l'obtenir, on observe que l'on a, pour toute paire $(x, z) \in F \times F$ et tout \'el\'ement $a \in F^{\times}$,
$$s u(z) \left( \begin{array}{cc} a & x \\ 0 & a^{-1} \end{array}\right) = \left(\begin{array}{cc} a^{-1} & 0 \\ 0 & a \end{array}\right) s u\left( \frac{a^{-1}z + x}{a}\right) \ .$$
On en d\'eduit donc que pour tout $b \in B_{S}$ et toute fonction $f \in V$, on a :
$$\forall y \in F, \  \Psi(b \cdot f)(u(y)) = \eta(b^{-1})(b \cdot \Psi(f))(u(y))\ . $$

Il nous suffit donc de d\'emontrer l'irr\'eductibilit\'e de la repr\'esentation $C^{\infty}_{c}(U)$ pour conclure.    Il est tout d'abord facile de voir que $U$ est la r\'eunion de ses sous-groupes ouverts compacts et que toute fonction $f \in C^{\infty}_{c}(U)$ est, par lissit\'e, \`a support dans un sous-groupe ouvert compact de $U$ (qui d\'epend bien s\^ur de $f$). Nous allons montrer que si l'on fixe un tel sous-groupe ouvert compact de $U$, sa fonction indicatrice engendre $C^{\infty}_{c}(U)$ comme repr\'esentation de $B_{S}$, ce qui nous permettra de conclure en utilisant la Proposition \ref{faitcrucial}.\\ 

Soient donc $U_{0}$ un sous-groupe ouvert compact de $U$ et $C^{\infty}_{c}(U_{0})$ l'espace des fonctions de $C^{\infty}_{c}(U)$ \`a support dans $U_{0}$. Il est clair que l'espace des vecteurs $U_{0}$-invariants de $C^{\infty}_{c}(U_{0})$ est de dimension 1 engendr\'e par la fonction indicatrice $1_{U_{0}}$. Comme $U_{0}$ est un pro-$p$-groupe, la Proposition \ref{faitcrucial} s'applique et toute sous-repr\'esentation non nulle de $C^{\infty}_{c}(U_{0})$ doit donc contenir $1_{U_{0}}$.\\
Montrons que $1_{U_{0}}$ engendre $C^{\infty}_{c}(U)$ (et pas seulement $C^{\infty}_{c}(U_{0})$) comme repr\'esentation de $B_{S}$. La famille de sous-groupes ouverts $U_{n} := \alpha_{0}^{n}U_{0}\alpha_{0}^{-n}$ ($n \in \NN$) est un syst\`eme fondamental de voisinages de $I_{2}$ dans $U$. De plus, l'action du groupe $U$ sur $1_{U_{n}}$ permet de r\'ecup\'erer toutes les fonctions indicatrices translat\'ees $1_{U_{n}u}$, $u \in U$. En \'ecrivant que $U_{0} = \alpha_{0}^{-n} U_{n} \alpha_{0}^{n}$ avec $\alpha_{0} \in T_{S}$, on obtient que l'action de $B_{S} = T_{S}U_{S}$ sur $1_{U_{0}}$ permet de r\'ecup\'erer tout l'espace $C^{\infty}_{c}(U)$, ce qui prouve son irr\'eductibilit\'e comme repr\'esentation de $B_{S}$.
\end{proof}
\begin{cor}
$Ind_{B_{S}}^{G_{S}}(\eta)$ est une repr\'esentation de $B_{S}$ de longueur 2.
\label{longueurG}
\end{cor}

\subsection{Fin de la preuve du Th\'eor\`eme \ref{sppal}}
Commen\c{c}ons par rappeler qu'il n'existe pas de mesure de Haar sur $F$ \`a valeurs dans $\FFF_{p}$ \cite[Proposition 6]{VigM2}\footnote{On rappelle qu'une mesure de Haar sur $F$ \`a valeurs dans $\FFF_{p}$ est une forme lin\'eaire non nulle $C_{c}^{\infty}(F) \to \FFF_{p}$ invariante par translations.}. Ceci implique que le module de Jacquet de $C^{\infty}_{c}(U)$ est nul et donc, gr\^ace \`a l'isomorphisme (\ref{isomcool}), que l'on a le r\'esultat suivant.
\begin{prop}
Le module de Jacquet de $V$ est nul.
\label{nonhaar}
\end{prop}
\begin{cor}
Le module de Jaquet de $Ind_{B_{S}}^{G_{S}}(\eta)$ est \'egal au caract\`ere $\eta$.
\end{cor}
\begin{proof}
La suite exacte courte de $B_{S}$-repr\'esentations
\begin{equation}
1 \longrightarrow V \longrightarrow Ind_{B_{S}}^{G_{S}}(\eta) \longrightarrow \eta \longrightarrow 1
\label{suitexB}
\end{equation}
fournit, apr\`es application du foncteur de Jacquet (qui est exact \`a droite) et utilisation de la Remarque \ref{JUcar}, la suite exacte de $T_{S}$-repr\'esentations
$$ V_{U} \longrightarrow \left( Ind_{B_{S}}^{G_{S}}(\eta)\right)_{U} \longrightarrow \eta \longrightarrow 1 \ ,$$
ce qui permet de conclure gr\^ace \`a la Proposition \ref{nonhaar}.
\end{proof}
\begin{cor}
Soit $\eta = \eta \otimes \mathbf{1}$ un caract\`ere lisse de $B_{S}$. Si $\mathbf{1}$ est un sous-quotient de la repr\'esentation $Ind_{B_{S}}^{G_{S}}(\eta)$ de $G_{S}$, alors $\eta = \mathbf{1}$.
\label{souquoJU}
\end{cor}

On termine la preuve du Th\'eor\`eme \ref{sppal} de la mani\`ere suivante : si l'on suppose que $Ind_{B_{S}}^{G_{S}}(\eta)$ est une repr\'esentation r\'eductible de $G_{S}$, elle admet alors n\'ecessairement un sous-quotient de dimension 1 (\`a cause du Corollaire \ref{longueurG} et de la suite exacte (\ref{suitexB})). La Proposition \ref{caracgs} assure que ce sous-quotient doit \^etre le caract\`ere trivial de $G_{S}$, ce qui implique que $\eta$ doit \^etre le caract\`ere trivial par exactitude \`a droite du foncteur de Jacquet (Corollaire \ref{souquoJU}).

\newpage
\section{Calculs de la Section \ref{decompo}}
\label{calcdecompo}
On d\'etaille ici le calcul donnant la surjectivit\'e de l'application du Th\'eor\`eme \ref{decomp}.\\
On sait \cite[Corollaires 4.1.1 et 4.1.4]{Br1} que $\pi(r,0,1)$ est une repr\'esentation irr\'eductible de $G$, ce qui implique en particulier que l'on a
\begin{equation}
\pi(r,0,1) = \langle G\cdot v_{r,0} \rangle = \langle G\cdot v_{r,\infty} \rangle \ .
\label{Grep}
\end{equation}
Par ailleurs, on dispose aussi des deux \'egalit\'es suivantes : 
$$ x^{r} = \sigma_{r}\Big(\left(\begin{array}{cc} 0 & 1 \\ -1 &0  \end{array} \right)\Big)(y^{r}) \ ;$$
$$y^{r} = \sigma_{r}\Big(\left( \begin{array}{cc} 0 & -1 \\ 1 & 0 \end{array}\right)\Big)(x^{r}) \ ;$$
qui assurent que l'on a les identit\'es suivantes dans c-ind$_{K\*Z}^{G}(\sigma_{r})$ (qui sont bien entendues valables a fortiori modulo $T_{r}$, i.e. dans $\pi(r,0,1)$) : 
\begin{equation}
[I_{2}, y^{r}] = [\left( \begin{array}{cc} 0 & -1 \\ 1 & 0 \end{array}\right), x^{r}] \in \ \langle K_{S}\cdot [I_{2}, x^{r}] \rangle \ ;
\label{ixer}
\end{equation}
\begin{equation}
[\alpha, x^{r}] = [\left( \begin{array}{cc} 0 & p^{-1} \\ -p & 0 \end{array}\right)\*\alpha, y^{r}] \in \ \langle G_{S}\cdot [\alpha, y^{r}] \rangle = \langle G_{S} \cdot [\beta , x^{r}] \rangle \ .
\label{grecer}
\end{equation}

\noindent Soit maintenant $F \in \pi(r,0,1)$. D'apr\`es (\ref{Grep}), $F$ est combinaison lin\'eaire d'\'el\'ements de la forme 
$$ g\cdot \overline{[I_{2}, x^{r}]} \  ;  \ h\cdot\overline{[\beta, x^{r}]}$$
avec $g$, $h \in G$. Soit $f$ est un tel \'el\'ement. On peut d\'ecomposer $g$ (respectivement $h$) en un produit fini de matrices $g_{1}, \ldots, g_{\ell}$ (resp. $h_{1}, \ldots, h_{\ell}$) de $G$ satisfaisant chacune l'un des quatre cas suivants : 
\begin{itemize}
\item ou bien $g$ ou $h$ est dans $G_{S}$, auquel cas $f$ est dans $\pi_{r,\infty}$ ou dans $\pi_{r,0}$ et il n'y a rien \`a d\'emontrer.\\
\item Ou bien $\det g$ (resp. $\det h$) est un carr\'e de $\mathbb{Q}_{p}^{*}$, auquel cas il existe $z \in Z$ tel que $g \in G_{S}\*Z$ (resp. $h \in G_{S}\*Z$). Comme $Z$ agit par des scalaires sur $\pi(r,0,1)$, on obtient de nouveau que $f$ appartient \`a $\pi_{r,\infty}$ (resp. $\pi_{r,0}$).\\
\item Ou bien $\det g = p$ (resp. $\det h = p$), ce qui implique alors que $g\*\alpha^{-1} \in G_{S}$ (resp. $h\*\alpha^{-1} \in G_{S}$). On peut donc \'ecrire $g = g_{0}\*\alpha$ (resp. $h = g_{0}\*\alpha$) avec $g_{0} \in G_{S}$, ce qui implique gr\^ace \`a (\ref{grecer}) que
$$f = g\cdot \overline{[I_{2}, x^{r}]} = g_{0}\cdot \overline{[\alpha, x^{r}]} \in \ \langle G_{S}\cdot \overline{[\alpha, y^{r}]} \rangle =  \langle G_{S} \cdot \overline{[\beta , x^{r}]} \rangle = \pi_{r,0} \ . $$
(resp. : gr\^ace \`a (\ref{ixer}) on a $f = h \cdot \overline{[\alpha, y^{r}]} = g_{0}\*\alpha^{2}\cdot \overline{[I_{2}, y^{r}]} \in \ \langle G_{S}\cdot \overline{[I_{2}, x^{r}]} \rangle = \pi_{r,\infty}$ car $\alpha^{2}$ est un \'el\'ement de $Z$.)\\
\item Ou bien $\det g = u \in \mathbb{Z}_{p}^{*}$ n'est pas un carr\'e de $\mathbb{Q}_{p}^{*}$, auquel cas on peut supposer, quitte \`a multiplier \`a gauche par un \'el\'ement de $K_{S}Z$ (ce qui pr\'eserve $\pi_{r, \infty}$ et $\pi_{r,0}$), que $g = \left( \begin{array}{cc} u & 0 \\ 0 & 1\end{array}\right)$. On a alors 
$$ \begin{array}{ccl} 
f = g\cdot \overline{[I_{2}, x^{r}]} & = & \overline{[\left( \begin{array}{cc} u & 0\\ 0 & 1 \end{array}\right), x^{r}]}\\
& = & \overline{[I_{2}, \sigma_{r}(\left( \begin{array}{cc} u & 0\\ 0 & 1 \end{array}\right))(x^{r})]}\\
& = & \overline{u}^{r}\*\overline{[I_{2}, x^{r}]}
\end{array}$$
o\`u $\overline{u}$ d\'esigne l'image modulo $p$ de $u$, de sorte que l'on a $f \in \pi_{r,\infty}$.\\
De m\^eme, si l'on suppose que $h$ est de d\'eterminant $u \in \mathbb{Z}_{p}^{*}$ qui n'est pas un carr\'e de $\mathbb{Q}_{p}^{*}$, alors on a aussi (en supposant de m\^eme que l'on prend $h = \left( \begin{array}{cc} u & 0 \\ 0 & 1\end{array}\right)$): 
$$  \begin{array}{ccl} 
f = h\cdot \overline{[\alpha, y^{r}]} & = & \overline{[\left( \begin{array}{cc} u & 0\\ 0 & 1 \end{array}\right)\*\alpha, y^{r}]}\\
& = & \overline{[\alpha, \sigma_{r}(\left( \begin{array}{cc} u & 0\\ 0 & 1 \end{array}\right))(y^{r})]}\\
& = & \overline{[\alpha, y^{r}]}
\end{array} $$
de sorte que $f$ appartient \`a $\pi_{r,0}$.
\end{itemize}
Une r\'ecurrence imm\'ediate sur $\ell$ permet alors de conclure que $f$ est effectivement contenu dans $\pi_{r,0}$ ou dans $\pi_{r,\infty}$, et donc que $F$ est finalement d\'ecomposable comme somme d'un \'el\'ement de $\pi_{r,0}$ et d'un \'el\'ement de $\pi_{r,\infty}$, ce qui prouve la surjectivit\'e recherch\'ee.


\end{document}